\documentclass[american,english]{amsart}
\usepackage[T1]{fontenc}
\usepackage[latin9]{inputenc}
\usepackage{babel}
\usepackage{amstext}
\usepackage{amsthm}
\usepackage{amssymb}
\usepackage{graphicx}
\usepackage{amsfonts}
\usepackage{bbm}

\numberwithin{equation}{section}
\numberwithin{figure}{section}
\theoremstyle{plain}
\newtheorem{thm}{\protect\theoremname}[section]
  \theoremstyle{plain}
  \newtheorem{lem}[thm]{\protect\lemmaname}
  \theoremstyle{plain}
  \newtheorem{prop}[thm]{\protect\propositionname}
  \theoremstyle{remark}
  \newtheorem{rem}[thm]{\protect\remarkname}
  \theoremstyle{plain}
  \newtheorem{conjecture}[thm]{\protect\conjecturename}
  \theoremstyle{plain}
  \newtheorem{cor}[thm]{\protect\corollaryname}

\DeclareMathOperator{\diag}{diag}

\DeclareMathOperator{\rank}{rank}
\DeclareMathOperator{\Mat}{Mat}
\DeclareMathOperator{\spn}{span}

\providecommand{\Id}{\mathbbm{1}}

  \addto\captionsamerican{\renewcommand{\conjecturename}{Conjecture}}
  \addto\captionsamerican{\renewcommand{\corollaryname}{Corollary}}
  \addto\captionsamerican{\renewcommand{\lemmaname}{Lemma}}
  \addto\captionsamerican{\renewcommand{\propositionname}{Proposition}}
  \addto\captionsamerican{\renewcommand{\remarkname}{Remark}}
  \addto\captionsamerican{\renewcommand{\theoremname}{Theorem}}
  \addto\captionsenglish{\renewcommand{\conjecturename}{Conjecture}}
  \addto\captionsenglish{\renewcommand{\corollaryname}{Corollary}}
  \addto\captionsenglish{\renewcommand{\lemmaname}{Lemma}}
  \addto\captionsenglish{\renewcommand{\propositionname}{Proposition}}
  \addto\captionsenglish{\renewcommand{\remarkname}{Remark}}
  \addto\captionsenglish{\renewcommand{\theoremname}{Theorem}}
  \providecommand{\conjecturename}{Conjecture}
  \providecommand{\corollaryname}{Corollary}
  \providecommand{\lemmaname}{Lemma}
  \providecommand{\propositionname}{Proposition}
  \providecommand{\remarkname}{Remark}
\providecommand{\theoremname}{Theorem}

\begin{document}



\title[Exotic Cluster Structures on $SL_{n}$ with BD Data of Minimal Size,
II]{Exotic Cluster Structures on $SL_{n}$ with Belavin--Drinfeld Data
of Minimal Size, II. Correspondence between cluster structures and
BD triples }

\date{November, 2015}

\author{Idan Eisner}

\address{Department of Mathematics, University of Haifa, 199 Abba Khoushy
Ave., Mount Carmel 3498838, Haifa, Israel}

\email{eisner@math.haifa.ac.il}
\begin{abstract}
Using the notion of compatibility between Poisson brackets and cluster
structures in the coordinate rings of simple Lie groups, Gekhtman
Shapiro and Vainshtein conjectured a correspondence between the two.
Poisson Lie groups are classified by the Belavin--Drinfeld classification
of solutions to the classical Yang Baxter equation. For any non trivial
Belavin--Drinfeld data of minimal size for $SL_{n}$, the companion
paper constructed a cluster structure with a locally regular initial
seed, which was proved to be compatible with the Poisson bracket associated
with that Belavin--Drinfeld data.

This paper proves the rest of the conjecture: the corresponding upper
cluster algebra $\overline{\mathcal{A}}_{\mathbb{C}}(\mathcal{C})$
is naturally isomorphic to $\mathcal{O}\left(SL_{n}\right)$, the
torus determined by the BD triple generates theaction
of $(\mathbb{C}^{*})^{2k_{T}}$ on $\mathbb{C}\left(SL_{n}\right)$,
and the correspondence between Belavin--Drinfeld classes and cluster
structures is one to one. 
\end{abstract}

\maketitle

\section{Introduction}

Since cluster algebras were introduced in \cite{FZ1}, a natural question
was the existence of cluster structures in the coordinate rings of
a given algebraic variety $V$. Partial answers were given for Grassmannians
$V=Gr_{k}\left(n\right)$ \cite{Scott} and double Bruhat cells \cite{BFZ}.
If $V=\mathcal{G}$ is a simple Lie group, one can extend the cluster
structure found in the double Bruhat cell to one in $\mathcal{O}\left(\mathcal{G}\right)$.
The compatibility of cluster structures and Poisson brackets, as characterized
in \cite{GSV1} suggested a connection between the two: given a Poisson
bracket, does a compatible cluster structure exist? Is there a way
to find it?

In the case that $V=\mathcal{G}$ is a simple complex Lie group, R-matrix
Poisson brackets on $\mathcal{G}$ are classified by the Belavin--Drinfeld
classification of solutions to the classical Yang Baxter equation
\cite{BDsolCYBE}. Given a solution of that kind, a Poisson bracket
can be defined on $\mathcal{G}$, making it a Poisson--Lie group.

The Belavin--Drinfeld (BD) classification is based on pairs of isometric
subsets of simple roots of the Lie algebra $\mathfrak{g}$ of $\mathcal{G}$.
The trivial case when the subsets are empty corresponds to the standard
Poisson bracket on $\mathcal{G}$ . It has been shown in \cite{gekhtman2012cluster}
that extending the cluster structure introduced in \cite{BFZ} from
the double Bruhat cell to the whole Lie group yields a cluster structure
that is compatible with the standard Poisson bracket. This led to
naming this cluster structure ``standard'', and trying to find other
cluster structures, compatible with brackets associated with non trivial
BD subsets. The term ``exotic'' was suggested for these non standard
structures \cite{gekhtman2013arxiv}.

Gekhtman, Shapiro and Vainshtein conjectured a bijection between BD
classes and cluster structures on simple Lie groups \cite{gekhtman2012cluster,gekhtman2013exotic}.
According to the conjecture, for a given BD class for $\mathcal{G}$,
there exists a cluster structure on $\mathcal{G}$, with rank determined
by the BD data. This cluster structure is compatible with the associated
Poisson bracket. The conjecture also states that the structure is
regular, and that the upper cluster algebra coincides with the ring
of regular functions on $\mathcal{G}$. The conjecture was proved
for the standard case and for $\mathcal{G}=SL_{n}$ with $n<5$ in
\cite{gekhtman2012cluster}. The Cremmer--Gervais case, which in some
sense is the ``furthest'' from the standard one, was proved in \cite{gekhtman2013exotic}.
It was also found to be true for all possible BD classes for $SL_{5}$
\cite{eisner2014SL5}.

This paper considers the conjecture for $SL_{n}$ when the BD data
is of minimal size, i.e., the two subsets contain exactly one simple
root. In the companion paper \cite{eisner2015part1} the first part
of the conjecture was proved: starting with two such subsets $\left\{ \alpha\right\} $
and $\left\{ \beta\right\} $, an algorithm was given constructing
a set $\mathcal{B}_{\alpha\beta}$ of functions that served as an
initial cluster. Defining an appropriate quiver $Q_{\alpha\beta}$
(or an exchange matrix $\tilde{B}_{\alpha\beta}$), one can obtain
an initial seed for a cluster structure on $SL_{n}$. This structure
is locally regular and it is compatible with the Poisson bracket associated
with the BD data $\left\{ \alpha\right\} \mapsto\left\{ \beta\right\} $. 

This paper completes the proof of the conjecture: the bijection between
cluster structures and BD classes of this type is established, the
upper cluster algebra is proved to be naturally isomorphic to the
ring of regular functions on $SL_{n}$, and a description of a global
toric action is given.

\section{Background }

\subsection{Cluster structures}

Let $\{z_{1},\ldots,z_{m}\}$ be a set of independent variables, and
let $S$ denote the ring of Laurent polynomials generated by $z_{1},\ldots,z_{m}$
- 
\[
S=\mathbb{Z}\left[z_{1}^{\pm1},\ldots,z_{m}^{\pm1}\right].
\]
 (Here and in what follows $z^{\pm1}$ stands for $z,z^{-1}$). The
\emph{ambient field} $\mathcal{F}$ is the field of rational functions
in $n$ independent variables (distinct from $z_{1},\ldots,z_{m}$),
with coefficients in the field of fractions of $S$.

A \emph{seed} (of geometric type) is a pair $(\textbf{x},\tilde{B})$,
where $\textbf{x}=(x_{1},\ldots,x_{n})$ is a transcendence basis
of $\mathcal{F}$ over the field of fractions of $S$, and $\tilde{B}$
is an $n\times(n+m)$ integer matrix whose principal part $B$ (that
is, the $n\times n$ matrix formed by columns $1,\ldots,n$) is skew-symmetric.
The set $\textbf{x}$ is called a \emph{cluster}, and its elements
$(x_{1},\ldots,x_{n})$ are called \emph{cluster variables}. Set $x_{n+i}=z_{i}$
for $i\in[1,m]$ (we use the notation $[a,b]$ for the set of integers
$\left\{ a,a+1,\ldots,b\right\} .$ When $a=1$ we write just $\left[m\right]$
for the set $\left[1,m\right]$). The elements $x_{n+1},\ldots,x_{n+m}$
are called \emph{stable variables} (or \emph{frozen variables}). The
set $\tilde{\textbf{x}}=(x_{1},\ldots,x_{n},x_{n+1},\ldots,x_{n+m})$
is called an \emph{extended cluster}. The square matrix $B$ is called
the \emph{exchange matrix}, and $\tilde{B}$ is called the \emph{extended
exchange matrix}. We sometimes denote the entries of $\tilde{\mathcal{B}}$
by $b_{ij}$, or say that $\tilde{B}$ is skew-symmetric when the
matrix $B$ has this property.

Let $\Sigma=(\tilde{\mathbf{x}},\tilde{B})$ be a seed. The set ${\tilde{\mathbf{x}}_{k}=(\tilde{\mathbf{x}}\setminus\{x_{k}\})\cup\left\{ x'_{k}\right\} }$
is called the adjacent cluster in direction $k\in\left[n\right]$,
where $x'_{k}$ is defined by the \emph{exchange relation 
\begin{equation}
x_{k}\cdot x'_{k}=\prod_{b_{kj}>0}x_{j}^{b_{kj}}+\prod_{b_{kj}<0}x_{j}^{-b_{kj}}.\label{eq:ExRltn}
\end{equation}
 } \emph{A matrix mutation} $\mu_{k}\left(\tilde{B}\right)$\emph{
}of $\tilde{B}$ in direction $k$ is defined by 
\[
b'_{ij}=\begin{cases}
-b_{ij} & \text{ if }i=k\text{ or }j=k\\
b_{ij}+\frac{1}{2}\left(\left|b_{ik}\right|b_{kj}+b_{ik}\left|b_{kj}\right|\right) & \text{ otherwise.}
\end{cases}
\]
 Seed mutation in direction $k$ is then defined $\mu_{k}\left(\Sigma\right)=(\tilde{\mathbf{x}}_{k},\mu_{k}(\tilde{B})).$

Two seeds are said to be mutation equivalent if they can be connected
by a sequence of seed mutations.

Given a seed $\Sigma=(\textbf{\ensuremath{\mathbf{x}}},\tilde{B})$,
the \emph{cluster structure} $\mathcal{C}(\Sigma)$ (sometimes denoted
$\mathcal{C}(\tilde{B})$, if $\mathbf{\mathbf{x}}$ is understood
from the context) is the set of all seeds that are mutation equivalent
to $\Sigma$. The number $n$ of rows in the matrix $\tilde{B}$ is
called the \emph{rank} of $\mathcal{C}$.

Let $\Sigma$ be a seed as above, and $\mathbb{A}=\mathbb{Z}\left[x_{n+1},\ldots,x_{n+m}\right]$.
The \emph{cluster algebra} $\mathcal{A}=\mathcal{A}(\mathcal{C})=\mathcal{A}(\tilde{B})$
associated with the seed $\Sigma$ is the $\mathbb{A}$-subalgebra
of $\mathcal{F}$ generated by all cluster variables in all seeds
in $\mathcal{C}(\tilde{B})$. The \emph{upper cluster algebra }$\mathcal{\overline{A}}=\mathcal{\overline{A}}(\mathcal{C})=\mathcal{\overline{A}}(\tilde{B})$\emph{
}is the intersection of the rings of Laurent polynomials over $\mathbb{A}$
in cluster variables taken over all seeds in $\mathcal{C}(\tilde{B})$.
The famous \emph{Laurent phenomenon} \cite{FZ2} claims the inclusion
$\mathcal{A}(\mathcal{C})\subseteq\mathcal{\overline{A}}(\mathcal{C})$.

A useful tool to prove that some function belongs to the upper cluster
algebra is the following Lemma \cite[Lemma 8.3]{gekhtman2013exotic}:
\begin{lem}
\label{lem:InAGSV}Let $\mathcal{C}$ be a cluster structure of geometric
type and $\overline{\mathcal{A}}$ be the corresponding upper cluster
algebra. Suppose $\frac{M_{1}}{f_{1}^{m_{1}}}=\frac{M_{2}}{f_{2}^{m_{2}}}=M$
for $M_{1},M_{2}\in\overline{\mathcal{A}}$, $m_{1},m_{2}\in\mathbb{N}$
and coprime cluster variables $f_{1}\neq f_{2}$. Then $M\in\overline{\mathcal{A}}$. 
\end{lem}
It is sometimes convenient to describe a cluster structure $\mathcal{C}(\tilde{B})$
in terms of its \emph{quiver} $Q(\tilde{B}):$ it is a directed graph
with $n+m$ nodes labeled $x_{1},\ldots,x_{n+m}$ (or just $1,\ldots,n+m$),
and an arrow pointing from $x_{i}$ to $x_{j}$ with weight $b_{ij}$
if $b_{ij}>0$.

Let $V$ be a quasi-affine variety over $\mathbb{C}$, $\mathbb{C}\left(V\right)$
be the field of rational functions on $V$, and $\mathcal{O}\left(V\right)$
be the ring of regular functions on $V$. Let $\mathcal{C}$ be a
cluster structure in $\mathcal{F}$ as above, and assume that $\left\{ f_{1},\ldots,f_{n+m}\right\} $
is a transcendence basis of $\mathbb{C}\left(V\right)$. Then the
map $\varphi:x_{i}\to f_{i}$, $1\leq i\leq n+m$, can be extended
to a field isomorphism $\varphi:\mathcal{F}_{\mathbb{C}}\to\mathbb{C}(V)$,
with $\mathcal{F}_{\mathbb{C}}=\mathcal{F}\otimes\mathbb{C}$ obtained
from $\mathcal{F}$ by extension of scalars. The pair$\left(\mathcal{C},\varphi\right)$
is then called a cluster structure in $\mathbb{C}\left(V\right)$
(or just a cluster structure on $V$), and the set $\left\{ f_{1},\ldots,f_{n+m}\right\} $
is called an extended cluster in $\left(\mathcal{C},\varphi\right)$.
Sometimes we omit direct indication of $\varphi$ and just say that
$C$ is a cluster structure on $V$. A cluster structure $\left(\mathcal{C},\varphi\right)$
is called \emph{regular} if $\varphi\left(x\right)$ is a regular
function for any cluster variable $x$, and a seed $\Sigma$ is called
\emph{locally regular }if all the cluster variables in $\Sigma$ and
in all the adjacent seeds are regular functions. The two algebras
defined above have their counterparts in $\mathcal{F}_{\mathbb{C}}$
obtained by extension of scalars; they are denoted $\mathcal{A}_{\mathbb{C}}$
and $\overline{\mathcal{A}}_{\mathbb{C}}$. If, moreover, the field
isomorphism $\varphi$ can be restricted to an isomorphism of $\mathcal{A}_{\mathbb{C}}$
(or $\overline{\mathcal{A}}_{\mathbb{C}}$) and $\mathcal{O}\left(V\right)$,
we say that $\mathcal{A}_{\mathbb{C}}$ (or $\overline{\mathcal{A}}_{\mathbb{C}}$)
is \emph{naturally isomorphic} to $\mathcal{O}\left(V\right)$.

The following statement is a weaker analogue of Proposition 3.37 in
\cite{GSV}:
\begin{prop}
\label{prop:ACNatIsoO(V)}Let $V$ be a Zariski open subset in $\mathbb{C}^{n+m}$
and $\left(\mathcal{C}=\mathcal{C}\left(\tilde{B}\right),\varphi\right)$
be a cluster structure in $\mathbb{C}\left(V\right)$ with $n$ cluster
and $m$ stable variables such that \end{prop}
\begin{enumerate}
\item $\rank\tilde{B}=n$;\label{prop:ACNIOVCond1} 
\item there exists an extended cluster $\tilde{\mathbf{x}}=\left(x_{1},\ldots,x_{n+m}\right)$
in $\mathcal{C}$ such that $\varphi\left(x_{i}\right)$ is regular
on $V$ for $i\in\left[n+m\right]$;\label{prop:ACNIOVCond2} 
\item for any cluster variable $x'_{k},\ k\in\left[n\right]$, obtained
by applying the exchange relation \eqref{eq:ExRltn} to $\tilde{\mathbf{x}}$,
$\varphi\left(x'_{k}\right)$ is regular on $V$;\label{prop:ACNIOVCond3} 
\item for any stable variable $x_{n+i}$, $i\in\left[m\right]$, the function
$\varphi\left(x_{n+i}\right)$ vanishes at some point of $V$;\label{prop:ACNIOVCond4} 
\item each regular function on $V$ belongs to $\varphi\left(\overline{\mathcal{A}}_{\mathbb{C}}\left(\mathcal{C}\right)\right)$.\label{prop:ACNIOVCond5} 
\end{enumerate}
Then $\mathcal{C}$ is a regular cluster structure and $\overline{\mathcal{A}}_{\mathbb{C}}\left(\mathcal{C}\right)$
is naturally isomorphic to $\mathcal{O}\left(V\right)$.

\subsection{Compatible Poisson brackets\label{sub:compPoissonBrckts}}

Let $\left\{ \cdot,\cdot\right\} $ be a Poisson bracket on the ambient
field $\mathcal{F}$. Two elements $f_{1},f_{2}\in\mathcal{F}$ are
\emph{log canonical }if there exists a rational number $\omega_{f_{1},f_{2}}$
such that $\left\{ f_{1},f_{2}\right\} =\omega_{f_{1},f_{2}}f_{1}f_{2}$.
A set $F\subseteq\mathcal{F}$ is called a log canonical set if every
pair $f_{1},f_{2}\in F$ is log canonical.

A cluster structure $\mathcal{C}$ in $\mathcal{F}$ is said to be
\emph{compatible} with the Poisson bracket $\left\{ \cdot,\cdot\right\} $
if every cluster is a log canonical set with respect to $\left\{ \cdot,\cdot\right\} $.
In other words, for every cluster $\mathbf{x}$ and every two cluster
variables $x_{i},x_{j}\in\mathbf{\tilde{x}}$ there exists $\omega_{ij}$
s.t. 
\begin{equation}
\left\{ x_{i},x_{j}\right\} =\omega_{ij}x_{i}x_{j}
\end{equation}

The skew symmetric matrix $\Omega^{\mathbf{\tilde{x}}}=(\omega_{ij})$
is called the \emph{coefficient matrix} of $\left\{ \cdot,\cdot\right\} $
(in the basis $\mathbf{\tilde{x}}$).

If $\mathcal{C}(\tilde{B})$ is a cluster structure of maximal rank
(i.e., $\rank\tilde{B}=n$), a complete characterization of all Poisson
brackets compatible with $\mathcal{C}(\tilde{B})$ is known (see \cite{GSV1},
and also \cite[Ch. 4]{GSV}). In particular, an immediate corollary
of Theorem 1.4 in \cite{GSV1} is:
\begin{prop}
\label{prop:PoissCompStruc}If $\rank\tilde{B}=n$ then a Poisson
bracket is compatible with $\mathcal{C}(\tilde{B})$ if and only if
its coefficient matrix $\Omega^{\tilde{\mathbf{x}}}$ satisfies $\tilde{B}\Omega^{\tilde{\mathbf{x}}}=\left[D\ 0\right]$,
where $D$ is a diagonal matrix. 
\end{prop}
A Lie group $\mathcal{G}$ with a Poisson bracket $\{\cdot,\cdot\}$
is called a \emph{Poisson--Lie group} if the multiplication map $\mu:\mathcal{G}\times\mathcal{G}\to\mathcal{G}$,
$\mu:(x,y)\mapsto xy$ is Poisson. That is, $\mathcal{G}$ with a
Poisson bracket $\{\cdot,\cdot\}$ is a Poisson--Lie group if 
\[
\{f_{1},f_{2}\}(xy)=\{\rho_{y}f_{1},\rho_{y}f_{2}\}(x)+\{\lambda_{x}f_{1},\lambda_{x}f_{2}\}(y),
\]
 where $\rho_{y}$ and $\lambda_{x}$ are, respectively, right and
left translation operators on $\mathcal{G}$.

A Poisson--Lie bracket on $SL_{n}$ can be extended to one on $GL_{n},$
with the determinant being a Casimir function. It will sometimes be
easier to discuss $GL_{n}$, and any statement can be restricted to
$SL_{n}$ by removing the determinant function.

Given a Lie group $\mathcal{G}$ with a Lie algebra $\mathfrak{g}$,
let $(\ ,\ )$ be a nondegenerate bilinear form on $\mathfrak{g}$,
and $\mathfrak{t}\in\mathfrak{g}\otimes\mathfrak{g}$ be the corresponding
Casimir element. For an element $r=\sum_{i}a_{i}\otimes b_{i}\in\mathfrak{g}\otimes\mathfrak{g}$
denote 
\[
\left[\left[r,r\right]\right]=\sum_{i,j}\left[a_{i},a_{j}\right]\otimes b_{i}\otimes b_{j}+\sum_{i,j}a_{i}\otimes\left[b_{i},a_{j}\right]\otimes b_{j}+\sum_{i,j}a_{i}\otimes a_{j}\otimes\left[b_{i},b_{j}\right]
\]
 and $r^{21}=\sum_{i}b_{i}\otimes a_{i}$.

The \emph{Classical Yang--Baxter equation (CYBE) }is 
\begin{equation}
\left[\left[r,r\right]\right]=0,\label{eq:CYBE}
\end{equation}
 an element $r\in\mathfrak{g}\otimes\mathfrak{g}$ that satisfies
\eqref{eq:CYBE} together with the condition 
\begin{equation}
r+r^{21}=\mathfrak{t}
\end{equation}
 is called a classical R-matrix.

A classical R-matrix $r$ induces a Poisson-Lie structure on $\mathcal{G}$:
choose a basis $\left\{ I_{\alpha}\right\} $ in $\mathfrak{g}$,
and denote by $\partial_{\alpha}$ (resp., $\partial'_{\alpha}$)
the left (resp., right) invariant vector field whose value at the
unit element is $I_{\alpha}$. Let $r=\sum_{\alpha,\beta}I_{\alpha}\otimes I_{\beta}$,
then 
\begin{equation}
\{f_{1},f_{2}\}_{r}=\sum_{\alpha,\beta}r_{\alpha,\beta}\left(\partial_{\alpha}f_{1}\partial_{\beta}f_{2}-\partial_{\alpha}^{\prime}f_{1}\partial_{\beta}^{\prime}f_{2}\right)\label{eq:sklnPB}
\end{equation}
 defines a Poisson bracket on $\mathcal{G}$. This is called the \emph{Sklyanin
bracket} corresponding to $r$.

In \cite{BDsolCYBE} Belavin and Drinfeld give a classification of
classical R-matrices for simple complex Lie groups: let $\mathfrak{g}$
be a simple complex Lie algebra with a fixed nondegenerate invariant
symmetric bilinear form $(\ ,\ )$. Fix a Cartan subalgebra $\mathfrak{h}$,
a root system $\Phi$ of $\mathfrak{g}$, and a set of positive roots
$\Phi^{+}$. Let $\Delta\subseteq\Phi^{+}$ be a set of positive simple
roots.

A Belavin--Drinfeld (BD) triple is two subsets $\Gamma_{1},\Gamma_{2}\subset\Delta$
and an isometry $\gamma:\Gamma_{1}\to\Gamma_{2}$ with the following
property: for every $\alpha\in\Gamma_{1}$ there exists $m\in\mathbb{N}$
such that $\gamma^{j}(\alpha)\in\Gamma_{1}$ for $j=0,\ldots,m-1$,
but $\gamma^{m}(\alpha)\notin\Gamma_{1}$. The isometry $\gamma$
extends in a natural way to a map between root systems generated by
$\Gamma_{1},\Gamma_{2}$. This allows one to define a partial ordering
on the root system: $\alpha\prec\beta$ if $\beta=\gamma^{j}\left(\alpha\right)$
for some $j\in\mathbb{N}$.

Select now root vectors $E_{\alpha}\in\mathfrak{g}$ that satisfy
$\left(E_{\alpha},E_{-\alpha}\right)=1$. According to the Belavin--Drinfeld
classification, the following is true (see, e.g., \cite[Ch. 3]{chriprsly}). 
\begin{prop}
(i) Every classical R-matrix is equivalent (up to an action of $\sigma\otimes\sigma$
where $\sigma$ is an automorphism of $\mathfrak{g}$) to 
\begin{equation}
r=r_{0}+\sum_{\alpha\in\Phi^{+}}E_{-\alpha}\otimes E_{\alpha}+\sum_{\begin{subarray}{c}
\alpha\prec\beta\\
\alpha,\beta\in\Phi^{+}
\end{subarray}}E_{-\alpha}\wedge E_{\beta}\label{eq:RmtxCons}
\end{equation}
 (ii) $r_{0}\in\mathfrak{h}\otimes\mathfrak{h}$ in \eqref{eq:RmtxCons}
satisfies 
\begin{equation}
\left(\gamma\left(\alpha\right)\otimes\Id\right)r_{0}+\left(\Id\otimes\alpha\right)r_{0}=0\label{eq:r0Cond1}
\end{equation}
 for any $\alpha\in\Gamma_{1}$, and 
\begin{equation}
r_{0}+r_{0}^{21}=\mathfrak{t}_{0},\label{eq:r0Cond2}
\end{equation}
 where $\mathfrak{t}_{0}$ is the $\mathfrak{h}\otimes\mathfrak{h}$
component of $\mathfrak{t}$. \\
(iii) Solutions $r_{0}$ to \eqref{eq:r0Cond1},\eqref{eq:r0Cond2}
form a linear space of dimension $k_{T}=\left|\Delta\setminus\Gamma_{1}\right|$. 
\end{prop}
Two classical R-matrices of the form \eqref{eq:RmtxCons} that are
associated with the same BD triple are said to belong to the same
Belavin--\emph{Drinfeld class. }The corresponding bracket defined
in~\eqref{eq:sklnPB} by an R-matrix $r$ associated with a triple
$T$ will be denoted by $\{\ ,\ \}_{T}$.
\begin{rem}
\label{rem:BDIsormrph}The BD data for $SL_{n}$ will be considered
up to the following two isomorphisms: the first reverses the direction
of $\gamma$ and transposes $\Gamma_{1}$ and $\Gamma_{2}$, while
the second one takes each root $\alpha_{j}$ to $\alpha_{\omega_{0}\left(j\right)}$,
where $\omega_{0}$ is the longest element in the Weyl group (which
in $SL_{n}$ is naturally identified with the symmetric group $S_{n-1}$).
These two isomorphisms correspond to the automorphisms of $SL_{n}$
given by $X\mapsto-X^{t}$ and $X\mapsto\omega_{0}X\omega_{0}$, respectively.
Since R-matrices are considered up to an action of $\sigma\otimes\sigma$,
from here on we do not distinguish between BD triples obtained one
from the other via these isomorphisms. We will also assume that in
the map $\gamma:\alpha_{i}\mapsto\alpha_{j}$ we always have $i<j$.
\end{rem}
Slightly abusing the notation, we sometime refer to a root $\alpha_{i}\in\Delta$
just as $i,$ and write $\gamma:i\mapsto j$ instead of $\gamma:\alpha_{i}\mapsto\alpha_{j}$.
For shorter notation, denote the BD triple
$\left(\left\{ \alpha\right\} ,\left\{ \beta\right\} ,\gamma:\alpha\mapsto\beta\right)$
by $T_{\alpha\beta}$, and naturally the corresponding Sklyanin bracket
will be $\left\{ \cdot,\cdot\right\} _{\alpha\beta}$ .

Given a BD triple $T$ for $\mathcal{G}$, write 
\[
\mathfrak{h}_{T}=\left\{ h\in\mathfrak{h}:\alpha(h)=\beta(h)\text{ if }\alpha\prec\beta\right\} ,
\]
 and define the torus $\mathcal{H}_{T}=\exp\mathfrak{h}_{T}\subset\mathcal{G}$.

The following conjecture was given by Gekhtman, Shapiro and Vainshtein
in \cite{gekhtman2012cluster}:
\begin{conjecture}
\label{Conj:GSV-BD-CS}Let $\mathcal{G}$ be a simple complex Lie
group. For any Belavin--Drinfeld triple $T=(\Gamma_{1},\Gamma_{2},\gamma)$
there exists a cluster structure $\mathcal{C}_{T}$ on $\mathcal{G}$
such that 
\begin{enumerate}
\item the number of stable variables is $2k_{T}$, and the corresponding
extended exchange matrix has a full rank. \label{Conj:NumStbVar} 
\item $\mathcal{C}_{T}$ is regular.\label{Conj:CTisregular} 
\item the corresponding upper cluster algebra $\overline{\mathcal{A}}_{\mathbb{C}}(\mathcal{C}_{T})$
is naturally isomorphic to $\mathcal{O}(\mathcal{G})$; \label{Conj:A(C)eqlsO(G)} 
\item the global toric action of $(\mathbb{C}^{*})^{2k_{T}}$ on $\mathbb{C}\left(\mathcal{G}\right)$
is generated by the action of $\mathcal{H}_{T}\otimes\mathcal{H}_{T}$
on $\mathcal{G}$ given by $\left(H_{1},H_{2}\right)\left(X\right)=H_{1}XH_{2}$
; \label{conj:global-toric} 
\item for any solution of CYBE that belongs to the Belavin--Drinfeld class
specified by $T$, the corresponding Sklyanin bracket is compatible
with $\mathcal{C}_{T}$; \label{Conj:Compatible} 
\item a Poisson--Lie bracket on $\mathcal{G}$ is compatible with $\mathcal{C}_{T}$
only if it is a scalar multiple of the Sklyanin bracket associated
with a solution of CYBE that belongs to the Belavin--Drinfeld class
specified by $T$. \label{Conj:a-Poisson--Lie-bracket} 
\end{enumerate}
\end{conjecture}
The conjecture was proved for the Belavin--Drinfeld class $\Gamma_{1}=\Gamma_{2}=\emptyset$.
This trivial triple corresponds to the standard Poisson--Lie bracket.
We call the cluster structures associated with the non-trivial Belavin--Drinfeld
data \emph{exotic}.

In the \emph{Cremmer--Gervais} case $\Gamma_{1}=\{\alpha_{1},\ldots,\alpha_{n-2}\}$,
$\Gamma_{2}=\{\alpha_{2},\ldots,\alpha_{n-1}\}$ and $\gamma:\alpha_{i}\mapsto\alpha_{i+1}$.
This case is, in some sense, ``the furthest'' from the standard
case, because here $\left|\Gamma_{1}\right|$ is maximal. Conjecture~\ref{Conj:GSV-BD-CS}
was proved for the Cremmer--Gervais case in \cite{gekhtman2013exotic}.
The conjecture is also true for all exotic cluster structures on $SL_{n}$
with $n\leq4$ \cite{gekhtman2012cluster} and for $SL_{5}$ \cite{eisner2015part1}.

One step towards a proof of the conjecture was this Proposition, from
\cite{gekhtman2012cluster}:
\begin{prop}
\label{prop:GSVReduct}Let $T=(\Gamma_{1},\Gamma_{2},\gamma)$ be
a BD triple. Suppose that assertions 1 and 4 of Conjecture \ref{Conj:GSV-BD-CS}
are valid and that assertion 5 is valid for one particular R-matrix
in the BD class specified by $T$ . Then 5 and 6 are valid for the
whole BD class specified by $T$.
\end{prop}

\section{Results}

For $SL_{n}$ with any $n\ge3$ and BD data $\{\alpha\}\mapsto\{\beta\}$,
part \ref{Conj:NumStbVar} of conjecture \ref{Conj:GSV-BD-CS} was
proved in \cite{eisner2015part1}. This paper completes the proof
by showing that parts \ref{Conj:CTisregular} -- \ref{Conj:a-Poisson--Lie-bracket}
are also true. 
\begin{thm}
For any $n\ge3$ and BD data $T=(\{\alpha\},\{\beta\},\alpha\mapsto\beta)$,
the upper cluster algebra $\overline{\mathcal{A}}_{\mathbb{C}}(\mathcal{C}_{T})$
on $SL_{n}$ is naturally isomorphic to $\mathcal{O}(SL_{n})$.\end{thm}
\begin{proof}
The proof relies on a Proposition \ref{prop:ACNatIsoO(V)}. Conditions
\ref{prop:ACNIOVCond1}, \ref{prop:ACNIOVCond2} and \ref{prop:ACNIOVCond3}
of the proposition hold, according to \cite{eisner2015part1}. To
satisfy Condition \ref{prop:ACNIOVCond4}, note that the Poisson bracket
$\left\{ \cdot,\cdot\right\} _{T}$ on $SL_{n}$ can be extended to
$\Mat(n)$ by requiring that the determinant is a Casimir function.
Clearly, on $\Mat(n)$ Condition \ref{prop:ACNIOVCond4} holds. Last,
we need that for any $n\ge3$ and BD data $T=(\{\alpha\},\{\beta\},\alpha\mapsto\beta)$,
all regular functions on $SL_{n}$ are in $\overline{\mathcal{A}}_{\mathbb{C}}(\mathcal{C}_{T})$.
To prove this, it suffices to show that for all $(i,j)\in[n]\times[n]$
the function $x_{ij}$ belongs to $\overline{\mathcal{A}}_{\mathbb{C}}(\mathcal{C}_{T})$.
This is true according to Theorem \ref{thm:xijbelongstoA}, and therefore
$\overline{\mathcal{A}}_{\mathbb{C}}(\mathcal{C}_{T})$ is naturally
isomorphic to $\mathcal{O}(SL_{n})$.
\end{proof}
The second result is statement \ref{conj:global-toric} of conjecture
\ref{Conj:GSV-BD-CS}:
\begin{thm}
\label{thm:The-global-toric}The global toric action of $(\mathbb{C}^{*})^{2k_{T}}$
on $\mathbb{C}\left(SL_{n}\right)$ is generated by the action of
$\mathcal{H}_{T}\otimes\mathcal{H}_{T}$ on $SL_{n}$ given by $\left(H_{1},H_{2}\right)\left(X\right)=H_{1}XH_{2}$. 
\end{thm}
which will be proved in Section \ref{sec:The-Toric-action}.

Now we can use Proposition \ref{prop:GSVReduct}, which implies that
assertions 5 and 6 of Conjecture \ref{Conj:GSV-BD-CS} are true.

\section{Construction of an initial seed\label{sec:ConstructionLCSeed}}

For a given BD data $T_{\alpha\beta}=\left(\left\{ \alpha\right\} ,\left\{ \beta\right\} ,\gamma:\alpha\mapsto\beta\right)$,
an initial seed for the cluster structure $\mathcal{C}_{\alpha\beta}$
is described in \cite{eisner2015part1}. The construction defines
a set of matrices $\mathcal{M}$ such that the set of all determinants
of these matrices forms the initial cluster. We repeat the main part
of it here:

For a matrix $X$ let $M_{ij}\left(X\right)$ be the maximal contiguous
submatrix of $X$ with $x_{ij}$ at the upper left hand corner. That
is, 
\begin{eqnarray*}
M_{ij}\left(X\right) & = & \begin{cases}
\left[\begin{array}{ccc}
x_{ij} & \cdots & x_{in}\\
\vdots &  & \vdots\\
x_{n-j+i,j} & \cdots & x_{n-j+i,n}
\end{array}\right] & \text{if }j>i\\
\\
\left[\begin{array}{ccc}
x_{ij} & \cdots & x_{i,n-i+j}\\
\vdots &  & \vdots\\
x_{nj} & \cdots & x_{n,n-i+j}
\end{array}\right] & \text{otherwise.}
\end{cases}
\end{eqnarray*}
Now define for every $(i,j)\in[n]\times[n]$,
\begin{equation}
f_{ij}(X)=\det M_{ij}(X).\label{eq:fijDefasDet}
\end{equation}

Let $X_{R}^{C}$ denote the submatrix of $X$ with rows in the set
$R$ and columns in $C$ (with $R,C\subseteq\left[n\right]$). Then
define two special families of matrices: for $1\leq j\leq\alpha$
and $i=n+j-\alpha$ set 
\[
\tilde{M}_{ij}(X)=\left[\begin{array}{cc}
X_{\left[i,n\right]}^{\left[j,\alpha+1\right]} & 0_{\left(n-i+1\right)\times\left(\mu-1\right)}\\
0_{\mu\times\left(n-i\right)} & X_{\left[1,\mu\right]}^{\left[\beta,n\right]}
\end{array}\right]
\]
with $\mu=n-\beta$, and for $1\le i\le\beta$ and $j=n+i-\beta$,
set 
\[
\tilde{M}_{ij}(X)=\left[\begin{array}{cc}
X_{\left[i,\beta+1\right]}^{\left[j,n\right]} & 0_{\left(n-j\right)\times\mu}\\
0_{\left(\mu-1\right)\times\left(n-j+1\right)} & X_{\left[\alpha,n\right]}^{\left[1,\mu\right]}
\end{array}\right],
\]
 and here $\mu=n-\alpha$. Note that these matrices are not block
diagonal: in the first case the number of columns in each of the two
blocks is greater than the number of rows by one, while in the second
case the number of rows in each block is greater than the number of
columns by one. 

Take the set of matrices $\left\{ M_{ij}\right\} _{i,j=1}^{n}$ and
whenever $i=n+j-\alpha$ or $j=n+i-\beta$ replace $M_{ij}$$\left(X\right)$
with $\tilde{M}_{ij}\left(X\right)$. This assures that for a fixed
pair $\left(i,j\right)\in[n]\times[n]$ there is still a unique matrix
in the set with $x_{ij}$ at the upper left corner. Denote this matrix
(either $M_{ij}(X)$ or $\tilde{M}_{ij}\left(X\right)$)
by $\overline{M}_{ij}$, and set 
\[
\varphi_{ij}=\det\overline{M}_{ij}.
\]
 Since $\varphi_{11}=\det X$ is constant on $SL_{n}$ we ignore it
and set the initial cluster to be 
$\mathcal{B}_{\alpha\beta}=\left\{ \varphi_{ij}|i,j\in\left[n\right]\right\} \setminus\left\{ \varphi_{11}\right\} $. 

To describe the initial quiver $Q_{\alpha\beta}$ start with the initial
quiver of the standard case on $SL_{n}$: the vertices are placed
on an $n\times n$ grid, with rows numbered from top to bottom and
columns numbered left to right. A vertex $(i,j)$ corresponds to the
cluster variable $f_{ij}$. The vertices in the first row and first
column are frozen. There are arrows from a vertex $(i,j)$ to vertices 
\begin{itemize}
\item $(i+1,j)$ if $i\neq n$
\item $(i,j+1)$ if $j\neq n$ 
\item $(i-1,j-1)$ if $i\neq1$ and $j\neq1$ 
\end{itemize}
The standard quiver for $SL_{5}$ is shown in Fig. \ref{fig:SL5StdQvr}
with squares representing frozen vertices and circles representing
mutable ones. 
\begin{figure}
\begin{centering}
\includegraphics[scale=0.45]{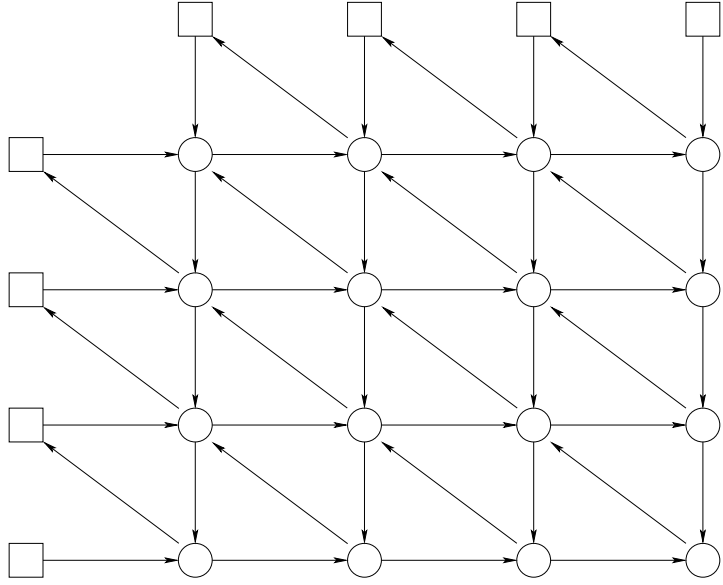} 
\end{centering}
\caption{The standard quiver for $SL_{5}$}
\label{fig:SL5StdQvr} 
\end{figure}

The quiver $Q_{\alpha\beta}(n)$ for the BD triple $T_{\alpha\beta}$
on $SL_{n}$ has similar form. It also has $n^{2}-1$ vertices on
an $n\times n$ grid, with same arrows as the standard one and the
following changes:
\begin{enumerate}
\item Vertices $\left(\alpha+1,1\right)$ and $\left(1,\beta+1\right)$
are not frozen. 
\item The arrows $\left(\alpha,1\right)\to\left(\alpha+1,1\right)$ and
$\left(1,\beta\right)\to\left(1,\beta+1\right)$ are added. 
\item The arrows $\left(n,\alpha+1\right)\to\left(1,\beta+1\right),\ \left(1,\beta+1\right)\to\left(n,\alpha\right)$
are added. 
\item The arrows $\left(\beta+1,n\right)\to\left(\alpha+1,1\right),\ \left(\alpha+1,1\right)\to\left(\beta,n\right)$
are added. 
\end{enumerate}
The example of $Q_{2\mapsto3}(5)$ on $SL_{5}$ is given in Figure
\ref{fig:SL5StdQvr}. The dashed arrows are the arrows that were added
to the standard quiver. 
\begin{figure}
\begin{centering}
\includegraphics[scale=0.45]{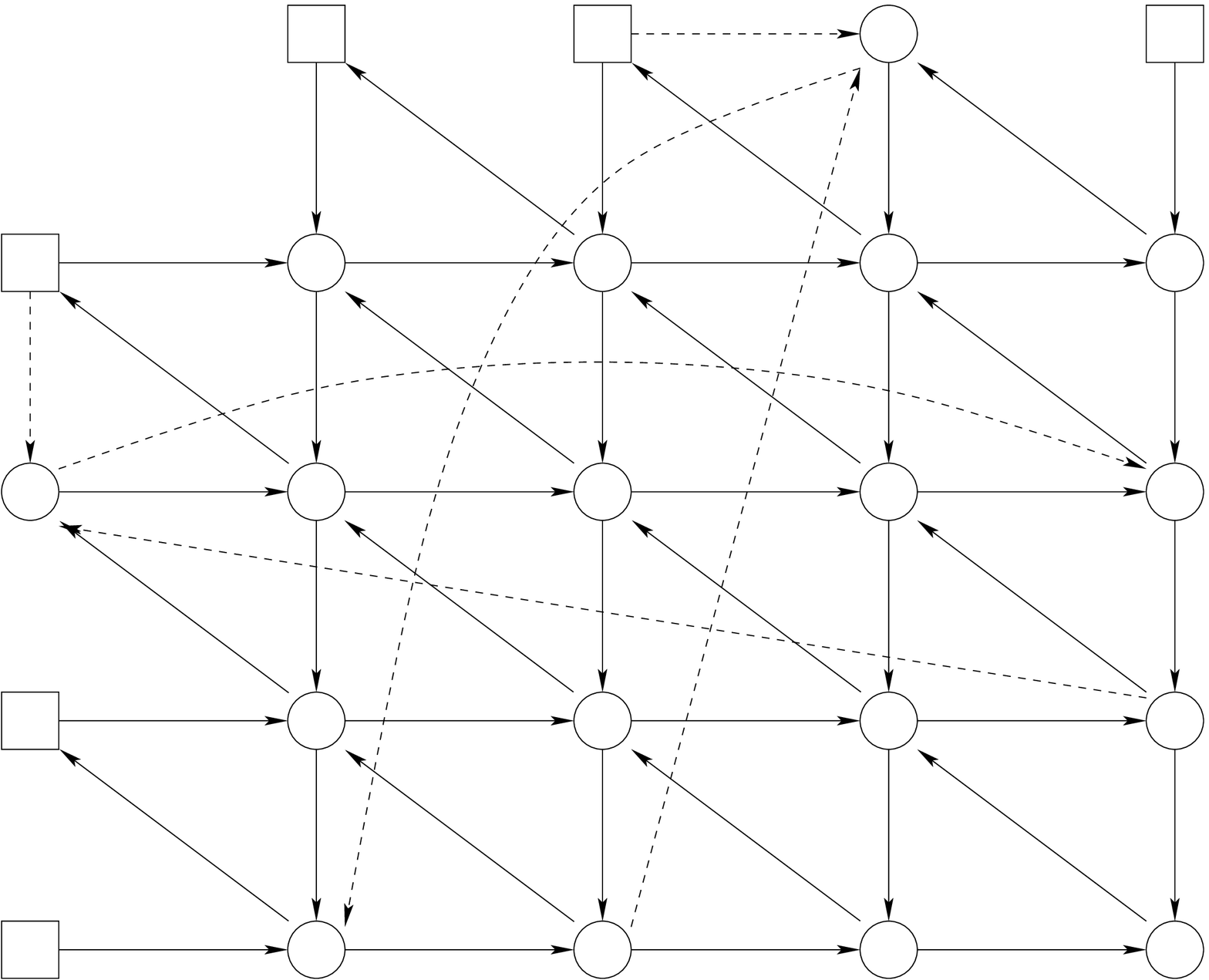} 
\par\end{centering}

\caption{The quiver $Q_{2\protect\mapsto3}$ on $SL_{5}$ }

\label{fig:SL52to3Qvr} 
\end{figure}

Associating every vertex $(i,j)$ of the quiver $Q_{\alpha\beta}$
with the cluster variable $\varphi_{ij}$ turns the pair $(\mathcal{B}_{\alpha\beta},Q_{\alpha\beta})$
into the initial seed for a cluster structure $\mathcal{C}_{\alpha\beta}$
on $SL_{n}$. This cluster structure is compatible with the bracket
$\{\cdot,\cdot\}_{\alpha\beta}$, and the
initial seed is locally regular. More
details can be found in \cite{eisner2015part1}.

In Section \ref{sec:The-natural-isomorphism} we will use the following
notations: for a function $f_{ij}=\det X_{[i,k]}^{[j,\ell]}$ we write
\begin{eqnarray*}
f_{ij}^{\rightarrow} & = & \det X_{[i,k]}^{[j,\ldots,\ell-1,\ell+1]}\\
f_{ij}^{\leftarrow} & = & \det X_{[i,k]}^{[j-1,j+1,\ldots,\ell]}\\
f_{ij}^{\uparrow} & = & \det X_{[i-1,i+1,\ldots,k}^{[j,\ell]}\\
f_{ij}^{\downarrow} & = & \det X_{[i,\ldots,k-1,k+1]}^{[j,\ell]}.
\end{eqnarray*}

\section{The natural isomorphism\label{sec:The-natural-isomorphism}}
\begin{thm}
$x_{ij}\in\overline{\mathcal{A}}_{\alpha\beta}$ for every $\left(i,j\right)\in\left[n\right]\times\left[n\right]$\label{thm:xijbelongstoA}\end{thm}
\begin{proof}
According to Lemma \ref{lem:xinandxnjInA} all $x_{in}$ and $x_{nj}$
are in $\overline{\mathcal{A}}_{\alpha\beta}$. If $\beta\neq n-1$
we use induction on $n$: according to Proposition \ref{prop:SeqSRdc}
$\overline{\mathcal{A}}_{\alpha\beta}(n-1)\subset\overline{\mathcal{A}}_{\alpha\beta}(n)$.
By the induction hypothesis, $x_{ij}\in\overline{\mathcal{A}}_{\alpha\beta}(n-1)$
for all $\left(i,j\right)\in\left[n-1\right]\times\left[n-1\right]$. 

If $\beta=n-1$ and $\alpha\neq1$ then $\overline{\mathcal{A}}_{\alpha\beta}$
is isomorphic to $\overline{\mathcal{A}}_{1,n-\alpha}$ (see Remark
\ref{rem:BDIsormrph}) and the argument above holds. Last, the case
$\alpha=1,\beta=n-1$ is covered by Lemma \ref{lem:xijinA1n-1}.\end{proof}
\begin{lem}
$x_{in}\in\overline{\mathcal{A}}_{\alpha\beta}$ and $x_{nj}\in\overline{\mathcal{A}}_{\alpha\beta}$
for all $i,j\in\left[n\right]$.\label{lem:xinandxnjInA} \end{lem}
\begin{proof}
$x_{nj}$ is a cluster variable in the initial cluster for every $j\neq\alpha$.
To see that $x_{n\alpha}\in\overline{\mathcal{A}}$, note that 
\[
\varphi_{n\alpha}=x_{n\alpha}f_{1,\beta+1}-x_{n,\alpha+1}f_{1,\beta+1}^{\leftarrow},
\]
 and therefore 
\begin{equation}
x_{n\alpha}=\frac{\varphi_{n\alpha}+x_{n,\alpha+1}f_{1,\beta+1}^{\leftarrow}}{f_{1,\beta+1}}.\label{eq:xnafirst}
\end{equation}
 Mutating at $(1,\beta+1)$ it is not hard to see that 
\[
\varphi_{1,\beta+1}^{\prime}=x_{n\alpha}f_{2,\beta+1}-x_{n,\alpha+1}f_{2,\beta+1}^{\leftarrow},
\]
 and so 
\begin{equation}
x_{n\alpha}=\frac{\varphi_{1,\beta+1}^{\prime}+x_{n,\alpha+1}f_{2,\beta+1}^{\leftarrow}}{f_{2,\beta+1}}.\label{eq:xnasecond}
\end{equation}
We can now combine \eqref{eq:xnafirst} and \eqref{eq:xnasecond}
into: 
\[
x_{n\alpha}=\frac{\varphi_{n\alpha}+x_{n,\alpha+1}f_{1,\beta+1}^{\leftarrow}}{f_{1,\beta+1}}=\frac{\varphi_{1,\beta+1}^{\prime}+x_{n,\alpha+1}f_{2,\beta+1}^{\leftarrow}}{f_{2,\beta+1}}.
\]
Clearly, $f_{1,\beta+1}$ and $f_{2,\beta+1}$ are cluster variables.
The functions $f_{1,\beta+1}^{\leftarrow}$ and $f_{2,\beta+1}^{\leftarrow}$
can be obtained from the initial cluster through the mutation sequence 

$\left((n,\beta+1),(n-1,\beta+1),\ldots(3,\beta+1),(2,\beta+1)\right)$,
as shown in \cite{eisner2015part1}\footnote{In \cite{eisner2015part1} it is shown that $f_{1,\beta+1}^{\leftarrow}$
and $f_{2,\beta+1}^{\leftarrow}$ are cluster variables of the \emph{standard}
structure on $SL_{n}.$ However, it is not difficult to see that this
is also true for our case, because locally the quiver looks the same
around the mutated vertices, and they correspond to the same cluster
variables.}. Lemma \ref{lem:InAGSV} now proves $x_{n\alpha}\in\overline{\mathcal{A}}$
.

In a similar way, for all $i\neq\beta$ the function $x_{in}$ is
a cluster variable in the initial seed, and therefore $x_{in}\in\overline{\mathcal{A}}$.
Symmetric arguments to those above yield 
\[
x_{\beta n}=\frac{\varphi_{\beta n}+x_{\beta+1,n}f_{\alpha+1,1}^{\uparrow}}{f_{\alpha+1,1}}
\]
 and 
\[
x_{\beta n}=\frac{\varphi_{\alpha+1,1}^{\prime}+x_{\beta+1,n}f_{\alpha+1,2}^{\uparrow}}{f_{\alpha+1,2}}.
\]
 The functions $f_{\alpha+1,1}^{\uparrow}$ and $f_{\alpha+1,2}^{\uparrow}$
are in $\overline{\mathcal{A}}$ (see \cite{eisner2015part1} again),
and $f_{\alpha+1,1}$ and $f_{\alpha+1,2}$ are cluster variables.
According to Lemma \ref{lem:InAGSV} $x_{\beta n}\in\overline{\mathcal{A}}$.
\end{proof}

\subsection{Mutation Sequence $S$ }
\begin{prop}
\label{prop:SeqSRdc}For $\beta\neq n-1$ there is a cluster mutation
sequence $S$ such that $S\left(Q_{\alpha\beta}\left(n\right)\right)$
contains a subquiver isomorphic to $Q_{\alpha\beta}\left(n-1\right)$.
The vertices of this subquiver can be rearranged on an $(n-1)\times(n-1)$
grid such that the vertex $(i,j)$ now corresponds to the cluster
variables $\varphi_{ij}$ as defined on $SL_{n-1}$.\end{prop}
\begin{proof}
Define $S$ as a composition $S=S_{v}\circ S_{h}$ of two sequences
$S_{v},S_{h}$ that are introduced hereinafter. Let $p_{h}^{\left(k\right)}$
be the sequence of quiver vertices starting at $\left(n,k\right)$
and moving diagonally up and to the left. That is, $p_{h}^{\left(k\right)}=\left(\left(n,k\right),\left(n-1,k-1\right),\ldots\right)$.
If the vertex $\left(n-k+1,1\right)$ is frozen, then it is the last
vertex of $p_{h}^{\left(k\right)}.$ If this is not the case, which
means $k=n-\alpha$, proceed with 
\[
p_{h}^{\left(k\right)}=\left(\left(n,k\right),\ldots,\left(\alpha+1,1\right),\left(\beta,n\right),\left(\beta-1,n-1\right),\ldots\right)
\]
 until hitting a frozen vertex (so the last vertex of $p_{h}^{\left(k\right)}$
must be frozen). Now let $S_{h}^{\left(k\right)}$ be the sequence
of mutations along the vertices of the path $p_{h}^{\left(n+1-k\right)}$,
excluding the last (frozen) vertex. 

Next, define a family of quivers $\left\{ Q_{h}^{\left(k\right)}\right\} _{t=o}^{n-1}$:
the first quiver is $Q_{h}^{\left(0\right)}=Q_{\alpha\beta}(n)$.
The quiver $Q_{h}^{\left(k+1\right)}$ is obtained from $Q_{h}^{\left(k\right)}$
by deleting the vertex $\left(n,n-k\right)$ and all the arrows incident
to it. If in $Q_{h}^{\left(k\right)}$there is an arrow connecting
$\left(n,n-k\right)$ to a vertex $\left(1,j\right)$, there are two
more changes: 
\begin{itemize}
\item an arrow $\left(n-1,n-k\right)\to\left(1,j\right)$ is added;
\item the arrow $\left(1,j\right)\to\left(n,n-k-1\right)$ is replaced by
$\left(1,j\right)\to\left(n-1,n-k-1\right)$.
\end{itemize}
The quivers $Q_{h}^{(3)}$ and $Q_{h}^{(4)}$ for BD triple $T_{24}$
(that is, $2\mapsto4$) on $SL_{6}$ are shown in Fig. \ref{fig:SL62to4Q3}
and Fig \ref{fig:SL62to4Q4}, respectively.

\begin{figure}
\begin{centering}
\includegraphics[scale=0.45]{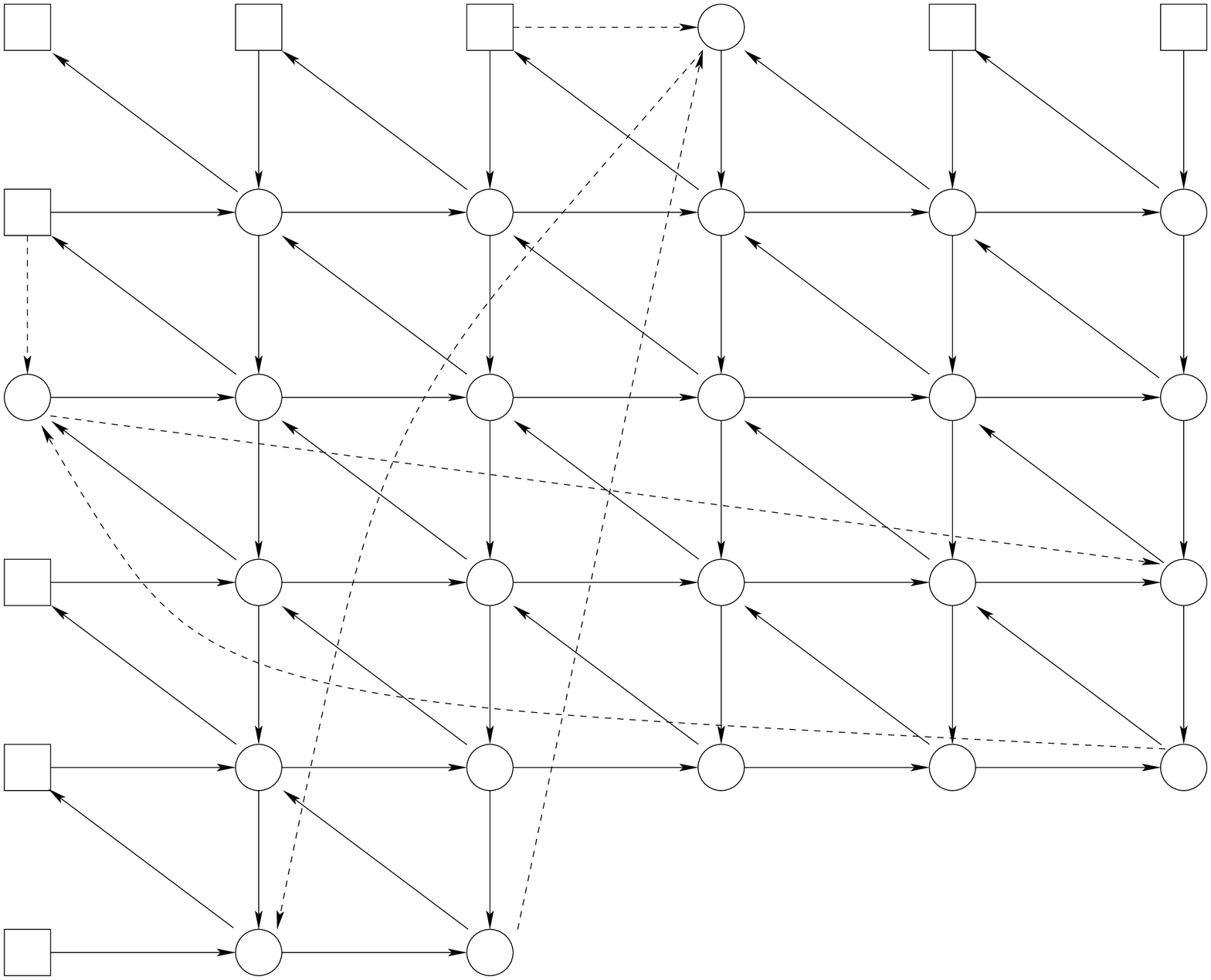} 
\par\end{centering}

\caption{
$Q_{h}^{(3)}$ for $2\protect\mapsto4$ on $SL_{6}$
}

\label{fig:SL62to4Q3} 
\end{figure}

\begin{figure}
\begin{centering}
\includegraphics[scale=0.45]{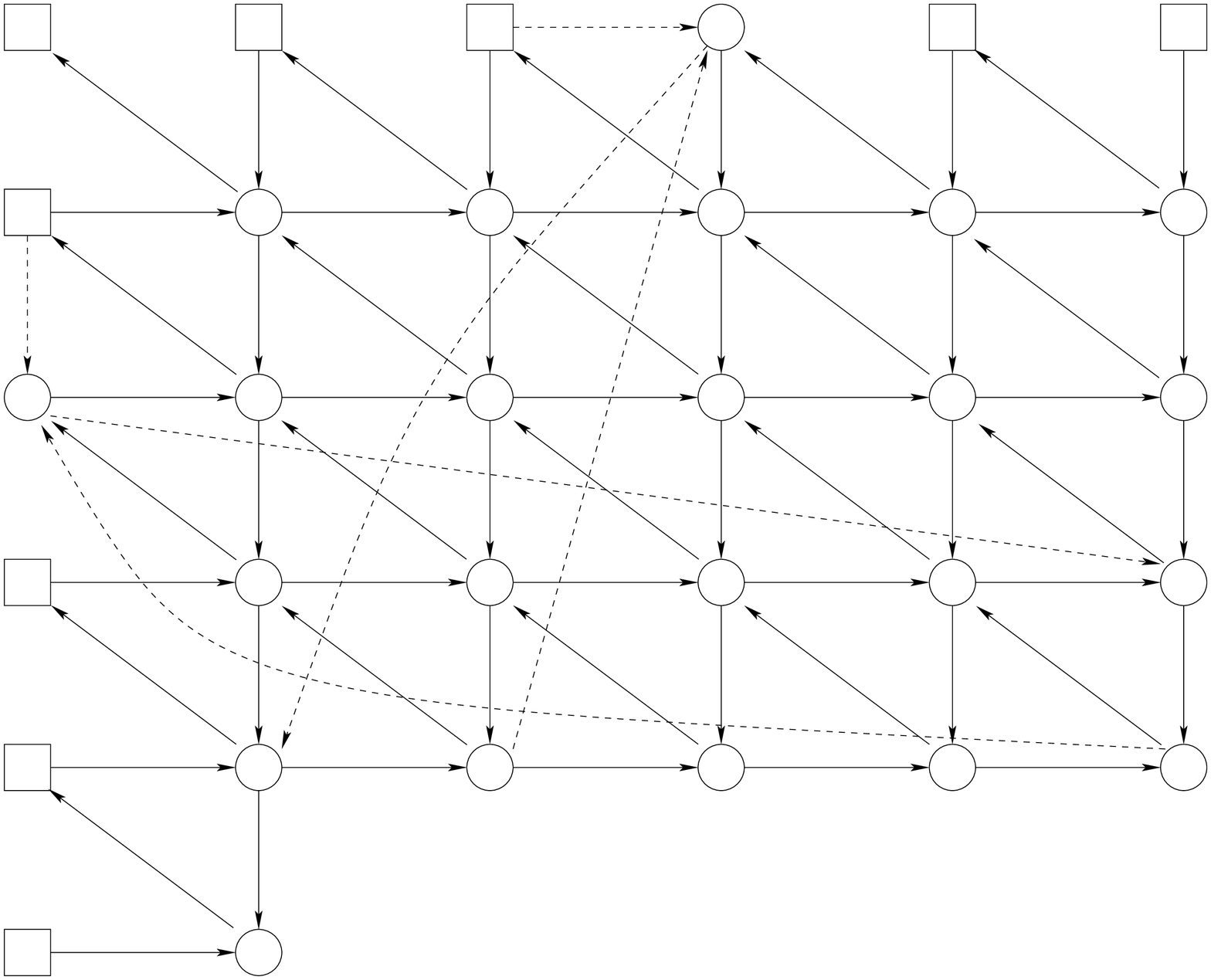} 
\par\end{centering}

\caption{
$Q_{h}^{(4)}$ for $2\protect\mapsto4$ on $SL_{6}$
}

\label{fig:SL62to4Q4} 
\end{figure}

Denote by $\sigma_{h}^{(k)}$ the following manipulations on a quiver
$Q$: 
\begin{enumerate}
\item Apply the mutation sequence $S_{h}^{\left(k\right)}$.
\item Freeze the last vertex of $S_{h}^{\left(k\right)}$.
\item Remove the last vertex of $p_{h}^{\left(n+1-k\right)}$ from the quiver.
\item Shift each vertex in $p_{h}^{\left(n+1-k\right)}$ to the position
of its successor (so $(i,j)$ is shifted to $(i-1,j-1)$ if $j\neq1$.
If $\left(\alpha+1,1\right)$ is in $p_{h}^{\left(n+1-k\right)}$,
it is shifted to $\left(\beta,n\right)$).
\end{enumerate}
We will show that after applying $S_{h}^{\left(k\right)}$ to $Q_{h}^{(k-1)}$,
the last vertex of $p_{h}^{\left(n+1-k\right)}$ has only one neighbor:
the last vertex of $S_{h}^{\left(k\right)}$. When the latter is frozen,
the last vertex of $p_{h}^{\left(n+1-k\right)}$ becomes isolated
(as arrows connecting frozen vertices are ignored), and can be removed
from the quiver.

We will assign a function $\psi_{ij}^{k}$ to every vertex $\left(i,j\right)$
of $Q_{h}^{\left(k\right)}$. For a function $f_{ij}=\det X_{[i,k]}^{[j,\ell]}$
write $f_{ij}^{(m)}=\det X_{[i,k-m]}^{[j,\ell-m]}$. Any function
$\varphi_{ij}$ can be written as determinant of some matrix $M$,
as defined in Section \ref{sec:ConstructionLCSeed}. In this case,
 $\varphi_{ij}^{(m)}$ will stand for the determinant of the submatrix
obtained from $M$ by deleting the last $m$ rows and last $m$ columns.
Now let 
\[
\psi_{ij}^{h}=\begin{cases}
f_{ij}^{(1)}f_{1,\beta+1}-f_{ij}^{(1)\rightarrow}f_{1,\beta+1}^{\leftarrow} & i=n-\alpha+1\\
\varphi_{ij}^{(1)} & \text{otherwise,}
\end{cases}
\]
set $\psi_{ij}^{0}=\varphi_{ij}$, and define 
\[
\psi_{ij}^{k+1}=\begin{cases}
\psi_{ij}^{h} & \left(i,j\right)\in p_{h}^{n+1-k}\\
\psi_{ij}^{k} & \text{otherwise,}
\end{cases}
\]

so the pair $\left(Q_{h}^{\left(k\right)},\left\{ \psi_{ij}^{k}\right\} \right)$
defines a seed that will be denoted $\Sigma_{h}^{\left(k\right)}$. 

In a similar way, define $p_{v}^{(k)},S_{v}^{(k)},\sigma_{v}(k)$
and $Q_{v}^{(k)}$: let $p_{v}^{(k)}$ denote the path $\left((k,n),(k-1,n-1),\ldots,(1,n-k+1)\right)$
if $(1,n-k+1)$ is a frozen vertex (i.e., if $k\neq n-\beta$). Otherwise
the path continues: $\left((n-\beta,n),(n-\beta-1,n-1),\ldots,(1,\beta+1),(n,\alpha),(n-1,\alpha-1)\ldots\right)$
until hitting a frozen vertex. The sequence $S_{v}^{(k)}$ is the
sequence of mutations at the vertices of $p_{v}^{(n-k)}$, excluding
the last (frozen) one. The quivers $Q_{v}^{(k)}$ are defined recursively:
set $Q_{v}^{(0)}=Q_{h}^{(n-1)}$ and let $Q_{v}^{(k+1)}$ be the quiver
obtained from $Q_{v}^{(k)}$ by removing the vertex $(n-k,n)$ and
all the arrows incident to it. If there is an arrow in $Q_{v}^{(k)}$
connecting $(n-k,n)$ to a vertex $(i,1)$ then in addition
\begin{itemize}
\item an arrow $(n-k,n-1)\to(\alpha+1,1)$ is added;
\item the arrow $(\alpha+1,1)\to(n-k-1,n)$ is replaced by $(\alpha+1,1)\to(n-k-1,n-1)$.
\end{itemize}
We now define $\sigma_{v}^{(k)}$ on a quiver $Q$:
\begin{enumerate}
\item Apply the mutation sequence $S_{v}^{\left(k\right)}$.
\item Freeze the last vertex of $S_{v}^{\left(k\right)}$.
\item Remove the last vertex of $p_{v}^{\left(n-k\right)}$ from the quiver.
\item Shift each vertex in $p_{h}^{\left(n-k\right)}$ to the position of
its successor (so $(i,j)$ is shifted to $(i-1,j-1)$ if $j\neq1$.
If $\left(1,\beta+1\right)$ is in $p_{h}^{\left(n-k\right)}$, it
is shifted to $\left(n,\alpha\right)$).
\end{enumerate}
Let $\sigma_{h}=\sigma_{h}^{(n-1)}\circ\cdots\circ\sigma_{h}^{(2)}\circ\sigma_{h}^{(1)}$
and $\sigma_{v}=\sigma_{v}^{(n-2)}\circ\cdots\circ\sigma_{v}^{(2)}\circ\sigma_{v}^{(1)}$
(and correspondingly, let $S_{h}=S_{h}^{(n-1)}\circ\cdots\circ S_{h}^{(2)}\circ S_{h}^{(1)}$
and $S_{v}=S_{v}^{(n-2)}\circ\cdots\circ S_{v}^{(2)}\circ S_{v}^{(1)}$).
According to Lemma \ref{lem:ShtakesSntoSh}, $\sigma_{h}\left(\Sigma{}_{\alpha\beta}(n)\right)=\Sigma_{h}^{(n-1)}$,
and according to Lemma \ref{lem:SeqSv-1} $\sigma_{v}\left(\Sigma_{h}^{(n-1)}\right)=\Sigma_{\alpha\beta}(n-1)$
. Therefore 
\[
\sigma_{v}\circ\sigma_{h}\left(\Sigma_{\alpha\beta}(n)\right)=\Sigma_{h(\alpha\beta)}(n-1),
\]
 which completes the proof.
\end{proof}
The following lemma uses the \emph{Desnanot--Jacobi
identity (see \cite[Th. 3.12]{Bressoud}): for a square matrix $A$,
denote by ``hatted'' subscripts and superscripts deleted rows and
columns, respectively. Then 
\begin{equation}
\det A\cdot\det A_{\hat{c}_{!},\hat{c}_{2}}^{\hat{r}_{1},\hat{r}_{2}}=\det A_{\hat{r}_{1}}^{\hat{c}_{1}}\cdot\det A_{\hat{r}_{2}}^{\hat{c}_{2}}-\det A_{\hat{r}_{2}}^{\hat{c}_{1}}\cdot\det A_{\hat{r}_{1}}^{\hat{c}_{2}}.\label{eq:DesJacId}
\end{equation}
 By adding an appropriate row, we get a similar result for a non square
matrix $B$ with number of rows greater by one than the number of
columns: 
\begin{equation}
\det B_{\hat{r_{1}}}\det B_{\hat{r}_{2}\hat{r}_{3}}^{\hat{c}_{1}}=\det B_{\hat{r}_{2}}\det B_{\hat{r}_{1}\hat{r}_{3}}^{\hat{c}_{1}}-\det B_{\hat{r}_{3}}\det B_{\hat{r}_{1}\hat{r_{2}}}^{\hat{c}_{1}},\label{eq:DesJacIdM}
\end{equation}
 and naturally, a similar identity can be obtained for a matrix with
number of columns greater by one than the number of rows. }
\begin{lem}
For any $\beta\neq n-1,$ the seed $\Sigma_{h}^{(n-1)}$ is obtained
from the initial seed $\Sigma_{\alpha\beta}(n)$ through the sequence
$\sigma_{h}=\sigma_{h}^{(n-1)}\circ\cdots\circ\sigma_{h}^{(2)}\circ\sigma_{h}^{(1)}$.\label{lem:ShtakesSntoSh}\end{lem}
\begin{proof}
Apply the sequence $S_{h}^{(1)}$ to the quiver $Q_{h}^{(0)}$: mutating
at $\left(n,n\right)$ removes the arrows $\left(n-1,n-1\right)\to\left(n-1,n\right)$
and $\left(n-1,n-1\right)\to\left(n,n-1\right)$, and reverses the
arrows incident to $\left(n,n\right)$ so now they are $\left(n,n\right)\to\left(n-1,n\right)$,
$\left(n,n\right)\to\left(n,n-1\right)$ and $\left(n-1,n-1\right)\to\left(n,n\right)$.
The next mutation at $\left(n-1,n-1\right)$ (shown in Fig. \eqref{fig:SkSndMut})
\begin{figure}
\begin{centering}
\includegraphics[scale=0.4]{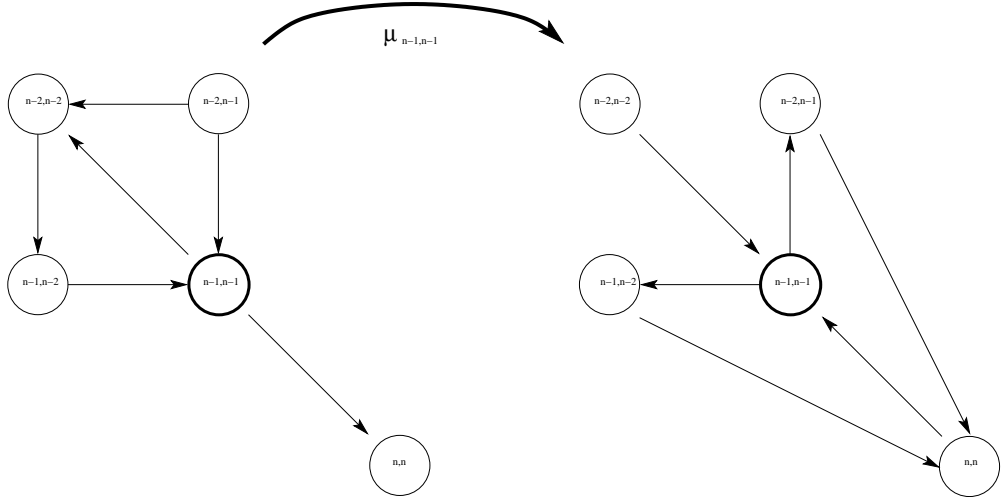} 
\par\end{centering}

\caption{
Mutation $S$ at $(n-1,n-1)$
}

\label{fig:SkSndMut} 
\end{figure}
then removes the arrows $\left(n-2,n-2\right)\to\left(n-2,n-1\right)$
and $\left(n-2,n-2\right)\to\left(n-1,n-2\right)$, and reverses arrows
so now we have $\left(n-1,n-1\right)\to\left(n-2,n-1\right)$, $\left(n-1,n-1\right)\to\left(n-1,n-2\right)$,
$\left(n-2,n-2\right)\to\left(n-1,n-1\right)$, and $\left(n,n\right)\to\left(n-1,n-1\right)$.
Use induction now to see that this mutation looks the same at each
vertex of the sequence. As a result, after the last mutation at $(2,2)$
the vertex $\left(1,1\right)$ has only $\left(2,2\right)$ as a neighbor.
Consequently, after freezing $\left(2,2\right)$ the vertex $\left(1,1\right)$
becomes isolated, and can not take part in any future mutations. It
is therefore removed from the quiver. It is not hard to see that $\sigma_{h}^{(1)}\left(Q_{h}^{(0)}\right)=Q_{h}^{(1)}$.

At each step of this sequence the exchange relation at $(i,i)$ is
\begin{equation}
\varphi_{ii}\varphi_{ii}^{\prime}=\varphi_{i-1,i-1}\varphi_{i+1,i+1}^{\prime}+\varphi_{i-1,i}\varphi_{i,i-1}.\label{eq:exrltatii}
\end{equation}

Assume (by induction) that $\varphi_{i+1,i+1}^{\prime}=\varphi_{ii}^{(1)}$
and write $A=X_{[i-1,n]}^{[i-1,n]}$ , with $\ell$ for the last row
(and column) of $A$. The exchange rule \eqref{eq:exrltatii} then
becomes 
\[
\varphi_{ii}\varphi_{ii}^{\prime}=\det A\det A_{\hat{1}\hat{\ell}}^{\hat{1}\hat{\ell}}+\det A_{\hat{\ell}}^{\hat{1}}\det A_{\hat{1}}^{\hat{\ell}}
\]
 and according to \eqref{eq:DesJacId} 
\[
\varphi_{ii}\varphi_{ii}^{\prime}=\det A_{\hat{1}}^{\hat{1}}\det A_{\hat{\ell}}^{\hat{\ell}}=\varphi_{ii}\varphi_{i-1,i-1}^{(1)},
\]

which means 
\[
\varphi_{ii}^{\prime}=\varphi_{i-1,i-1}^{(1)},
\]
 and this proves that $\sigma_{1}\left(\Sigma_{\alpha\beta}(n)\right)=\Sigma_{h}^{(1)}$.

We will now show that $\sigma_{h}^{(k)}\left(Q_{h}^{(k-1)}\right)=Q_{h}^{(k)}$
for every $k\in[n-1]$. If $k\notin\left\{ \alpha+1,n-\alpha\right\} $
then we use induction again to assume that locally at each step the
path $p_{h}^{\left(n+1-k\right)}$ is isomorphic to $p_{h}^{\left(n\right)}$
, so the quiver mutations are (locally) identical, and after the mutations
and the removal of the last vertex we end up with $Q_{h}^{(k)}$.
The mutation rule at $\left(n-m,n+1-k-m\right)$ of the sequence $S_{h}^{(k)}$
looks like the one at $\left(n-m,n-m\right)$ in the sequence $S_{h}^{(1)}$.
Put $i=n-m$ and $j=n+1-k-m$ so the exchange relation is 
\[
\varphi_{ij}\varphi_{ij}^{\prime}=\varphi_{i-1,j-1}\varphi_{i+1,i+1}^{\prime}+\varphi_{i-1,j}^{(1)}\varphi_{i,j-1}.
\]
Assume by induction $\varphi_{i+1,i+1}^{\prime}=\varphi_{ij}^{(1)},$
and write $A=X_{[i,n]}^{[j,\mu]}$ with $\mu=n-i+j$. We can now use
\eqref{eq:DesJacId} to obtain 
\[
\varphi_{ij}^{\prime}=\varphi_{ij}^{(1)}.
\]
Hence, for $k\notin\left\{ \alpha+1,n-\alpha\right\} $ we have 
\begin{equation}
\sigma_{h}^{(k)}\left(\Sigma_{h}^{(k-1)}\right)=\Sigma_{h}^{(k)}.\label{eq:sktakesSk1-toSk}
\end{equation}

If $k=\alpha+1$, then the above holds for the first part of the sequence
$S_{h}^{(k)}$, which is $\left(\left(n,n-\alpha\right),\ldots\left(\alpha+2,2\right)\right)$.
The sequence then continues with a mutation at $\left(\alpha+1,1\right)$
as illustrated in Fig. \ref{fig:S1ata+11}:

\begin{figure}
\begin{centering}
\includegraphics[scale=0.4]{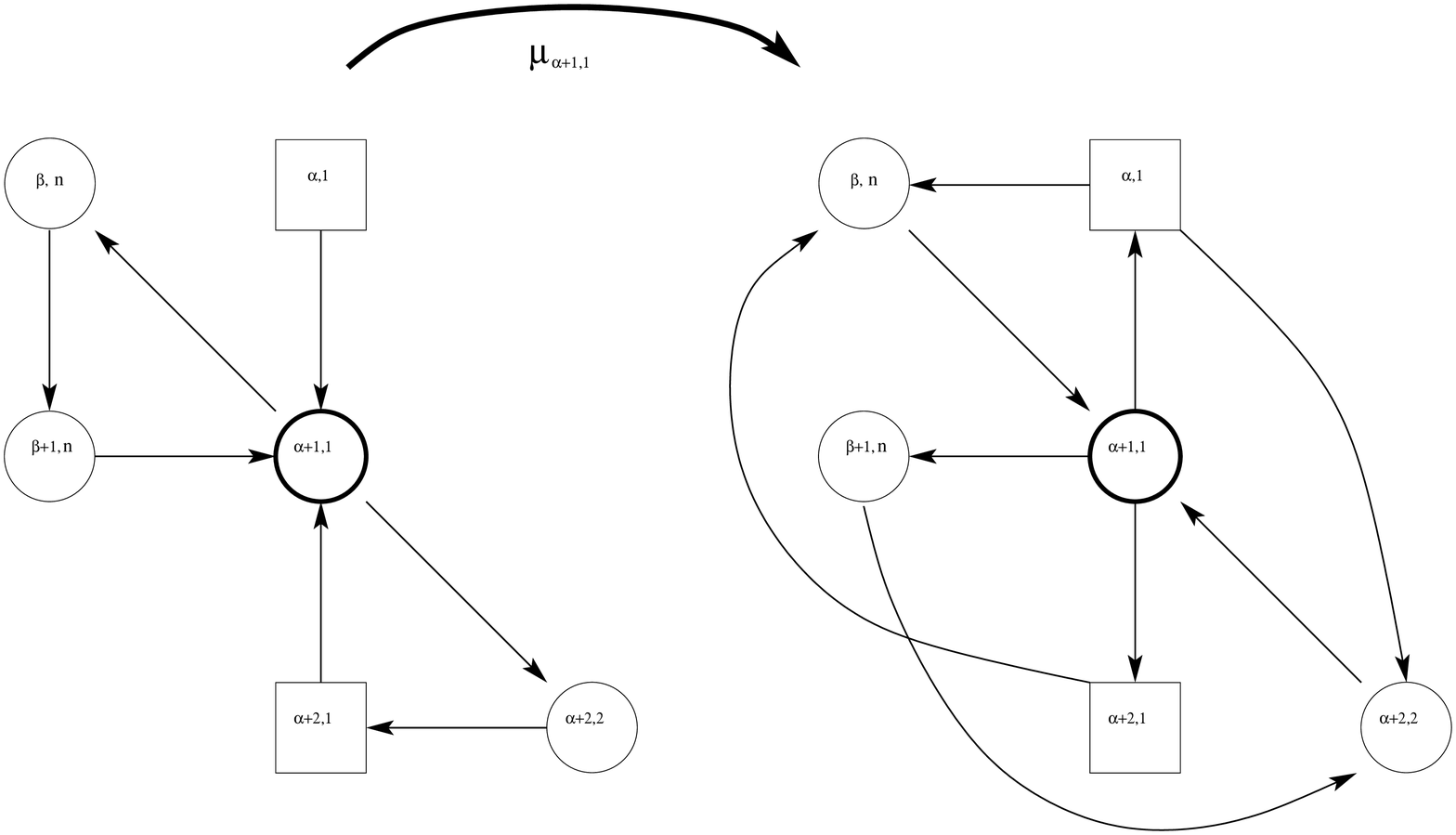} 
\par\end{centering}

\caption{
Mutation $S$at $(\alpha+1,1)$
}

\label{fig:S1ata+11} 
\end{figure}

the arrows $\left(\beta,n\right)\to\left(\beta+1,n\right)$ and $\left(\alpha+2,2\right)\to\left(\alpha+2,1\right)$
are removed, and three new arrows are added: $\left(\alpha,1\right)\to\left(\beta,n\right)$,
$\left(\alpha+2,1\right)\to\left(\beta,n\right)$ and $\left(\alpha,1\right)\to\left(\alpha+2,2\right)$.
All the arrows incident to $\left(\alpha+1,1\right)$ are inverted. 

The exchange relation is 
\begin{equation}
\varphi_{\alpha+1,1}\varphi_{\alpha+1,1}^{\prime}=\varphi_{\alpha+2,1}\varphi_{\alpha1}\varphi_{\beta+1,n}+\varphi_{\alpha+2,2}^{\prime}\varphi_{\beta n},\label{eq:exrltata+11S}
\end{equation}
and writing 
\[
A=\left[\begin{array}{ccccc}
x_{\beta n} & x_{\alpha1} & \cdots & \cdots & x_{\alpha\mu}\\
x_{\beta+1,n} & x_{\alpha+1,1}\\
0 & x_{\alpha+2,1} & \ddots\\
\vdots & \vdots &  & \ddots\\
0 & x_{n1} & \cdots &  & x_{n\mu}
\end{array}\right]
\]
with $\mu=n-\alpha$, one can use \eqref{eq:DesJacId} to obtain 
\[
\varphi_{\alpha+1,1}^{\prime}=\det\left[\begin{array}{ccccc}
x_{\beta n} & x_{\alpha1} & \cdots & \cdots & x_{\alpha,\mu-1}\\
x_{\beta+1,n} & x_{\alpha+1,1}\\
0 & x_{\alpha+2,1} & \ddots\\
\vdots & \vdots &  & \ddots\\
0 & x_{n-1,1} & \cdots &  & x_{n-1,\mu-1}
\end{array}\right]=\varphi_{\beta n}^{(1)}.
\]

The next mutations are at $\left(\beta-k,n-k\right)$ with $k=0,1,\ldots,\beta+2$
(for $k=0$ we label $\left(\alpha+1,1\right)$ as $\left(\beta+1,n+1\right)$).
The relevant part of the quiver is shown in Fig. \ref{fig:Smutatbn}.

\begin{figure}
\begin{centering}
\includegraphics[scale=0.4]{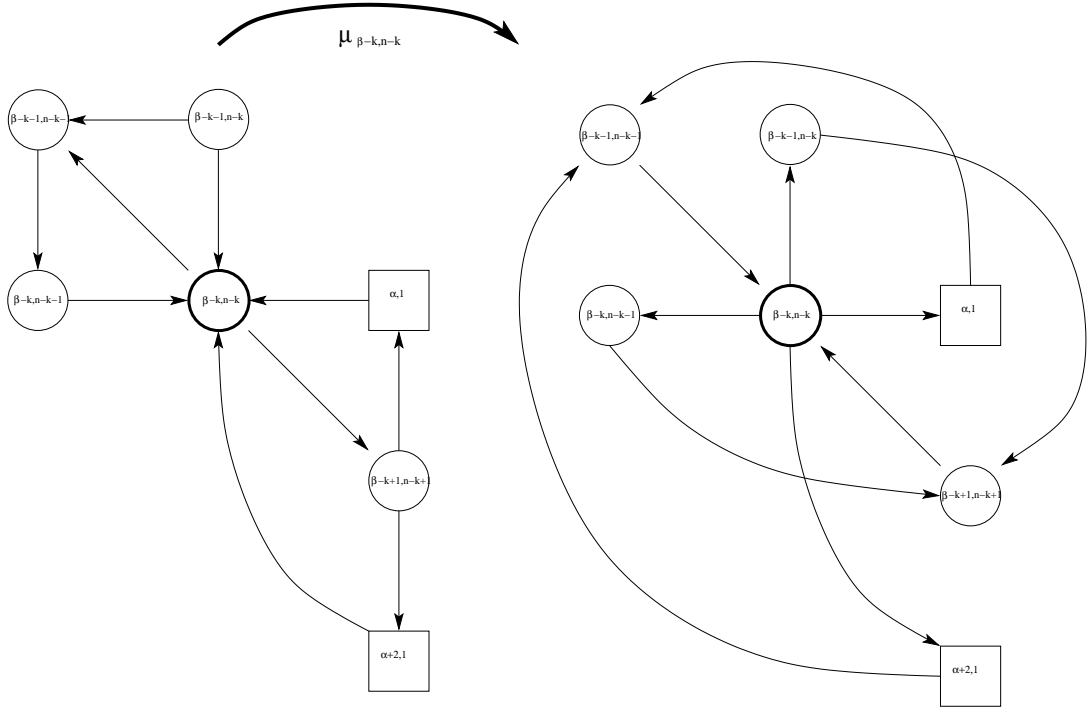} 
\par\end{centering}

\caption{
Mutation $S$ at $(\beta-k,n-k)$
}

\label{fig:Smutatbn} 
\end{figure}
 Four arrows -- $\left(\beta-k-1,n-k\right)\to\left(\beta-k-1,n-k-1\right)$,
$\left(\beta-k-1,n-k-1\right)\to\left(\beta-k,n-k-1\right)$, $\left(\beta-k-1,n-k-1\right)\to\left(\alpha,1\right)$
and $\left(\beta-k-1,n-k-1\right)\to\left(\alpha+2,1\right)$ are
removed, and four new arrows are added: $(\alpha,1)\to(\beta-k-1,n-k-1)$,
$(\alpha+2,1)\to(\beta-k-1,n-k-1)$, $(\beta-k-1,n-k)\to(\beta-k+1,n-k+1)$
and $(\beta-k,n-k-1)\to(\beta-k+1,n-k+1)$. In addition, all the arrows
incident to $(\beta-k,n-k)$ are inverted.

The exchange rule here (with $i=\beta-k,$ and $j=n-k$) is 
\begin{equation}
\varphi_{ij}\varphi_{ij}^{\prime}=\varphi_{i-1,j-1}\varphi_{i+1,i+1}^{\prime}+\varphi_{i-1,j}\varphi_{i,j-1}\varphi_{\alpha1}\varphi_{\alpha+2,1}.\label{eq:exrltatbnS}
\end{equation}

Assume by induction $\varphi_{i+1,i+1}^{\prime}=\varphi_{ij}^{(1)},$
and as in the previous cases, write 
\[
A=\left[\begin{array}{ccccccc}
x_{i-1,j-1} & \cdots & x_{i-1,n} & 0 & \cdots & \cdots & 0\\
\vdots & \ddots & \vdots & \vdots &  &  & \vdots\\
\vdots &  & x_{\beta n} & x_{\alpha1} & \cdots & \cdots & x_{\alpha\mu}\\
x_{\beta+1,j-1} & \cdots & x_{\beta+1,n} & x_{\alpha+1,1} &  &  & \vdots\\
0 & \cdots & 0 & \vdots & \ddots &  & \vdots\\
\vdots &  & \vdots &  &  & \ddots & \vdots\\
0 & \cdots & 0 & x_{n1} & \cdots & \cdots & x_{n\mu}
\end{array}\right]
\]
 and let $\ell$ denote the last row (and column) of $A$, and now
\eqref{eq:exrltatbnS} turns to 
\[
\varphi_{ij}\varphi_{ij}^{\prime}=\det A\det A_{\hat{1}\hat{\ell}}^{\hat{1}\hat{\ell}}+\det A_{\hat{1}}^{\hat{\ell}}\det A_{\hat{\ell}}^{\hat{1}}=\det A_{\hat{1}}^{\hat{1}}\det A_{\hat{\ell}}^{\hat{\ell}},
\]
 so that $\varphi_{ij}^{\prime}=\varphi_{i-1,j-1}^{(1)}$. This proves
\eqref{eq:sktakesSk1-toSk} for $k=\alpha+1$.

The last case is $k=n-\alpha$, with $S_{h}^{(k)}=\left(\left(n,\alpha+1\right),\left(n-1,\alpha\right),\ldots\right)$.
Fig. \ref{fig:S1atna+1} shows the mutation at $\left(n,\alpha+1\right)$:
\begin{figure}
\begin{centering}
\includegraphics[scale=0.45]{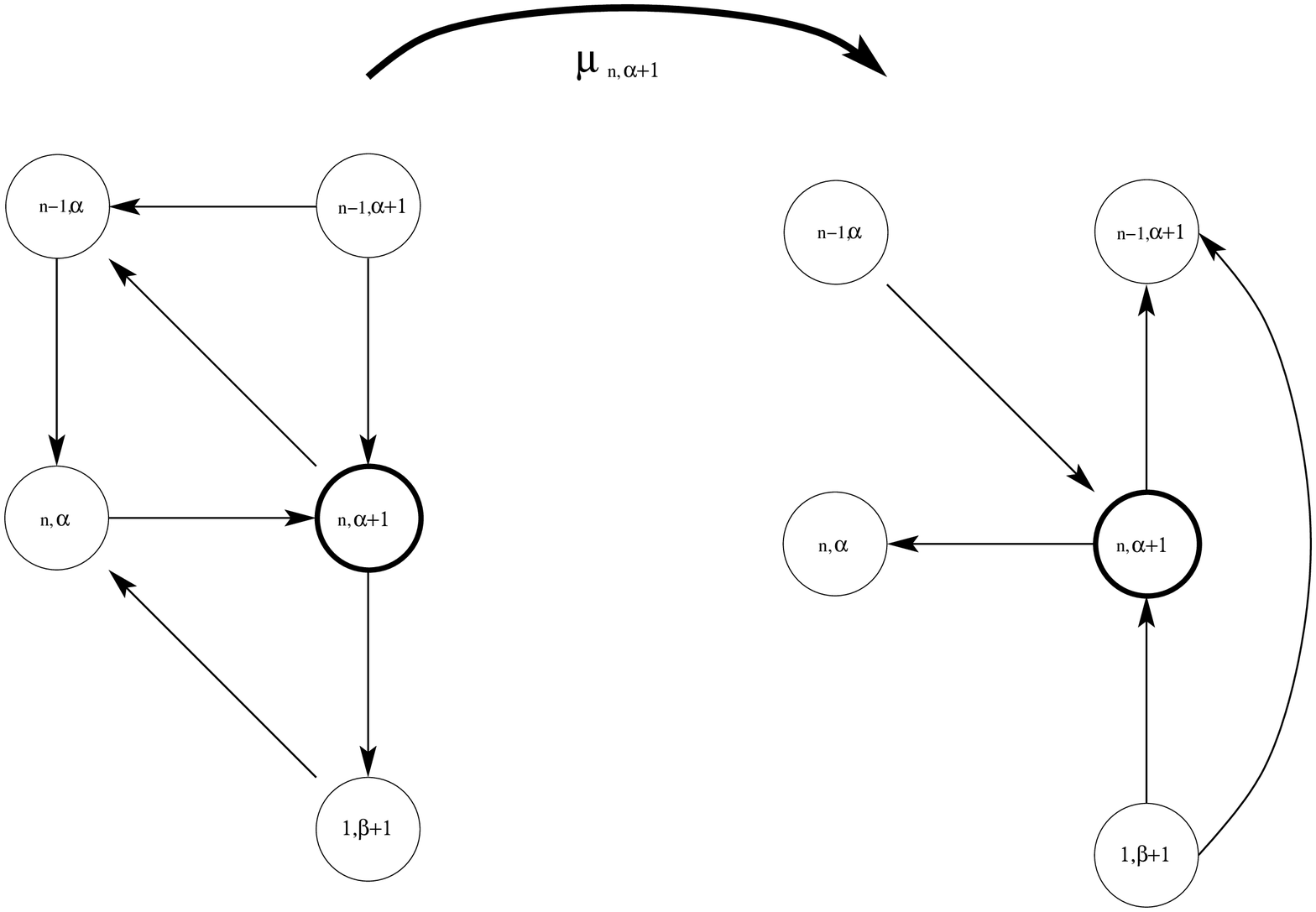} 
\par\end{centering}

\caption{
Mutation $S$at $\left(n,\alpha+1\right)$
}

\label{fig:S1atna+1} 
\end{figure}
 The arrows $(n-1,\alpha+1)\to(n-1,\alpha)$, $(n-1,\alpha)\to(n,\alpha)$
and $(1,\beta+1)\to(n,\alpha)$ are removed, while an arrow $(1,\beta+1)\to(n-1,\alpha+1)$
is added. The four arrows incident to $(n,\alpha+1$) are inverted. 

The exchange relation here is 
\[
\varphi_{n,\alpha+1}\varphi_{n,\alpha+1}^{\prime}=\varphi_{n\alpha}\varphi_{n-1,\alpha+1}^{(1)}+\varphi_{n-1,\alpha}\varphi_{1,\beta+1},
\]
 because the function associated with the vertex $(n-1,\alpha+1)$
is now $\varphi_{n-1,\alpha+1}^{(1)}$ (in the quiver $Q_{\alpha\beta}^{(k-1)}$).
Recall that $\varphi_{n\alpha}=x_{n\alpha}f_{1,\beta+1}-x_{n,\alpha+1}f_{1,\beta+1}^{\leftarrow}$,
so we can write 
\begin{eqnarray*}
\varphi_{n,\alpha+1}\varphi_{n,\alpha+1}^{\prime} & = 
& \left(x_{n\alpha}f_{1,\beta+1}-x_{n,\alpha+1}f_{1,\beta+1}^{\leftarrow}\right)x_{n-1,\alpha+1}\\
& &+\left|\begin{array}{cc}
x_{n-1,\alpha} & x_{n-1,\alpha+1}\\
x_{n\alpha} & x_{n,\alpha+1}
\end{array}\right|f_{1,\beta+1}\\
 & = & x_{n,\alpha+1}\left(x_{n-1,\alpha}f_{1,\beta+1}-x_{n-1,\alpha+1}f_{1,\beta+1}^{\leftarrow}\right)
\end{eqnarray*}
 and therefore 
\[
\varphi_{n,\alpha+1}^{\prime}=x_{n-1,\alpha}f_{1,\beta+1}-x_{n-1,\alpha+1}f_{1,\beta+1}^{\leftarrow}.
\]
 Proceeding along the path $S_{h}^{(k)}=\left(\left(n,\alpha+1\right),\left(n-1,\alpha\right),\ldots\right)$
, the result is 
\[
\varphi_{ij}^{\prime}=f_{i-1,j-1}^{(1)}f_{1,\beta+1}-f_{i-1,j-1}^{(1)\rightarrow}f_{1,\beta+1}^{\leftarrow}
\]
 for all $i=n-\alpha-1+j$. This shows that \eqref{eq:sktakesSk1-toSk}
holds for $k=n-\alpha$ as well.

We have shown that $\sigma_{h}^{(k)}\left(\Sigma{}_{h}^{(k-1)}\right)=\Sigma{}_{h}^{(k)}$
for $k\in[1,n]$. Thus, writing $\sigma_{h}=\sigma_{h}^{(n-1)}\cdots\circ\sigma_{h}^{(2)}\circ\sigma_{h}^{(1)}$
we have proved 
\[
\sigma_{h}\left(\Sigma_{\alpha\beta}\right)=\Sigma_{h(\alpha\beta)}.
\]
\end{proof}
\begin{lem}
The seed $\Sigma_{\alpha\beta}(n-1)$ is obtained from the initial
seed $\Sigma_{h}^{(n-1)}$ through the sequence $\sigma_{v}=\sigma_{v}^{(n-1)}\circ\cdots\circ\sigma_{v}^{(2)}\circ\sigma_{v}^{(1)}.$
\label{lem:SeqSv-1}\end{lem}
\begin{proof}
The proof is almost identical to the proof of Lemma \ref{lem:ShtakesSntoSh}.
\end{proof}

\subsection{Mutation sequences $T_{k}$}

Start with the mutation sequence $T_{1}$, mutating along columns
of the quiver $Q_{1\mapsto n-1}$, bottom to top, from the right column
to the left. That is, define
\begin{eqnarray*}
T_{1} & = & (\left(n,n\right),\left(n-1,n\right),\ldots,\left(1,n\right),\left(n,n-1\right),\left(n-1,n-1\right),\ldots,\\
 &  & \left(2,n-1\right),\ldots,\left(n,2\right),\left(n-1,2\right),\ldots,\left(2,2\right)),
\end{eqnarray*}
 and let
\begin{eqnarray*}
T_{m} & = & (\left(n,n\right),\ldots,\left(m+1,n\right),\left(n,n-1\right),\ldots\left(m+1,n-1\right),\ldots,\\
 &  & \left(n,m+1\right),\ldots,\left(m+1,m+1\right)).
\end{eqnarray*}
 This sequence has only columns $n,n-1,\ldots m+1$ and each column
only contains rows $n,n-1,\ldots m+1$. Note that $T_{m}\subset T_{m-1}$.
Set $\Sigma_{0}=\left(\mathcal{B}_{1,n-1},Q_{1,n-1}(n)\right)$ (the
initial seed) and define $\Sigma_{m}=\left(\mathcal{B}_{m},Q^{m}\right)$
as the seed obtained from $\Sigma_{m-1}$ through the sequence of
mutations $T_{m}$.  Let $\varphi_{ij}^{m}$ denote the function
associated with the vertex $\left(i,j\right)$ in the seed $\Sigma_{m}$,
and let 
\[
\tilde{f}_{ij}=f_{ij}f_{21}-f_{ij}^{\downarrow}f_{21}^{\uparrow},
\]
and 
\[
\tilde{f}_{ij}^{[1]}=f_{ij}^{(1)}f_{21}-f_{ij}^{(1)\downarrow}f_{21}^{\uparrow}.
\]
Note that if $\varphi=\tilde{f}_{ij}$ then $\varphi^{(1)}\neq\tilde{f}_{ij}^{[1]}$. 
\begin{prop}
\label{prop:SeedS1}In the seed $\Sigma_{1}$,  

1. for  $i,j\in\left\{ 2,\ldots,n\right\} $
the cluster variables take the form
\begin{eqnarray*}
\varphi_{ij}^{1} & = & \begin{cases}
f_{i-1,j-1}^{\left(1\right)} & j>i,\\
\tilde{f}_{i-1,j-1}^{[1]} & j\le i;
\end{cases}
\end{eqnarray*}
 
2. after freezing all the vertices $\left(2,j\right)$ and $\left(i,2\right)$,
and vertices $(1,n)$ and $(2,1)$, all the vertices $\left(1,j\right)$
and $\left(i,1\right)$ become isolated (connected only to frozen
vertices) and can therefore be ignored. The subquiver on vertices
$i,j\in\left\{ 2,\ldots,n\right\} $ is isomorphic to the standard
quiver $Q_{std}(n-1)$ on $SL_{n-1}$.\end{prop}
\begin{proof}
We look at the sequence $T_{1}$ step by step: start with mutating
the initial cluster at $\left(n,n\right)$. The relevant part of the
quiver is given in Fig. \ref{fig:T1atnn}. So the exchange relation
is 
\begin{eqnarray*}
\varphi_{nn}\varphi_{nn}^{1} & = & \varphi_{n-1,n-1}\cdot\varphi_{21}+\varphi_{n-1,n}\cdot\varphi_{n,n-1}\\
 & = & \left|\begin{array}{rr}
x_{n-1,n-1} & x_{n-1,n}\\
x_{n,n-1} & x_{nn}
\end{array}\right|\cdot f_{21}+\tilde{f}_{n-1,n}x_{n,n-1}
\end{eqnarray*}
\begin{figure}
\begin{centering}
\includegraphics[scale=0.5]{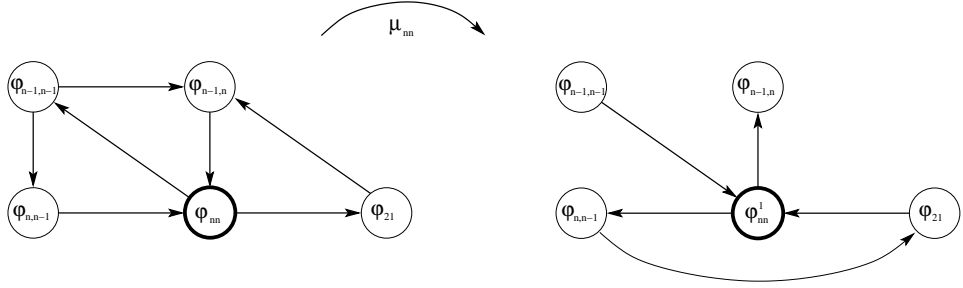} 
\par\end{centering}

\caption{
Mutation $T_{1}$ at vertex $\left(n,n\right)$
}

\label{fig:T1atnn} 
\end{figure}
and since $f_{n-1,n}=x_{n-1,n},$ it comes down to 
\[
\varphi_{nn}\varphi_{nn}^{1}=x_{nn}\left(x_{n-1,n-1}f_{21}-x_{n,n-1}f_{21}^{\uparrow}\right)
\]
which means 
\begin{equation}
\varphi_{nn}^{1}=x_{n-1,n-1}f_{21}-x_{n,n-1}f_{21}^{\uparrow}=\tilde{f}_{n-1,n-1}^{[1]}.
\end{equation}
The arrows $\left(n-1,n-1\right)\to\left(n,n-1\right),\left(n-1,n-1\right)\to\left(n-1,n\right)$
and $\left(2,1\right)\to\left(n-1,n\right)$ are removed, and an arrow
$\left(n,n-1\right)\to\left(2,1\right)$ is added. All the arrows
that touch $\left(n,n\right)$ are inverted (as shown in Fig. \ref{fig:T1atnn}).

Now mutate at $\left(n-1,n\right)$. The exchange relation is determined
by the partial quiver in Fig. \ref{fig:T1atkn}:
\begin{eqnarray*}
\varphi_{n-1,n}\varphi_{n-1,n}^{1} & = & \varphi_{n-2,n}\varphi_{nn}^{1}+\varphi_{n-2,n-2}\\
 & = & x_{n-2,n}\tilde{f}_{n-1,n-1}^{[1]}+\tilde{f}_{n-2,n-1}\\
 & = & x_{n-2,n-1}\tilde{f}_{n-1,n}
\end{eqnarray*}
\begin{figure}
\begin{centering}
\includegraphics[scale=0.5]{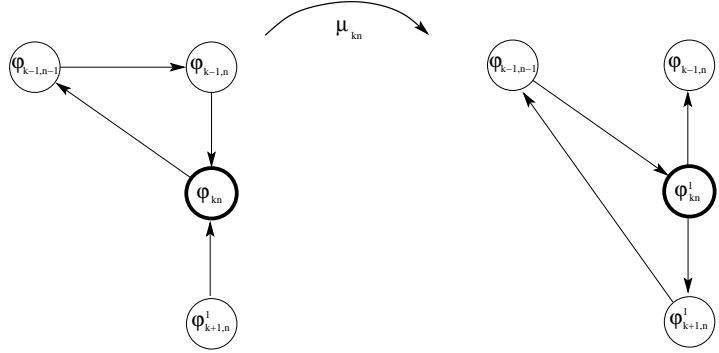} 
\par\end{centering}

\caption{
Mutation $T_{1}$ at vertex $\left(k,n\right)$
}

\label{fig:T1atkn} 
\end{figure}
and since we had $\varphi_{n-1,n}=\tilde{f}_{n-1,n}$, we get 
\begin{equation}
\varphi_{n-1,n}^{1}=x_{n-2,n-1}=f_{n-2,n-1}^{\left(1\right)}.
\end{equation}
The arrow $\left(k-1,n-1\right)\to\left(k-1,n\right)$ was removed,
and an arrow $\left(k+1,n\right)\to\left(k-1,n-1\right)$ was added.
All the arrows touching $\varphi_{kn}$ were inverted. 

The sequence continues mutating at $\left(k,n\right)$ with $k=n-2,n-3,\ldots,2$:
locally the quiver looks like in Fig. \ref{fig:T1atkn} and so 
\begin{eqnarray*}
\varphi_{kn}\varphi_{kn}^{1} & = & \varphi_{k+1,n}^{1}\varphi_{k-1,n}+\varphi_{k-1,n-1}\\
 & = & x_{k,n-1}x_{k-1,n}+\left|\begin{array}{rr}
x_{k-1,n-1} & x_{k-1,n}\\
x_{k,n-1} & x_{kn}
\end{array}\right|\\
 & = & x_{k-1,n-1}x_{kn},
\end{eqnarray*}
 and so 
\[
\varphi_{kn}^{1}=x_{k-1,n-1}=f_{k-1,n-1}^{\left(1\right)}.
\]
The last mutation on column $n$ is at the vertex $\left(1,n\right)$.
The exchange relation is determined by Fig. \ref{fig:T1atkn} with
$k=1$, so now $k-1$ should be replaced with $n$, and so 
\begin{eqnarray*}
\varphi_{1n}\varphi_{1n}^{1} & = & \varphi_{2n}^{1}\varphi_{n2}+\varphi_{n1}\\
 & = & x_{1,n-1}x_{n2}+\left|\begin{array}{rr}
x_{n1} & x_{n2}\\
x_{1,n-1} & x_{1n}
\end{array}\right|\\
 & = & x_{n1}x_{1n}.
\end{eqnarray*}
 Therefore 
\begin{equation}
\varphi_{1n}^{1}=x_{n1}.\label{eq:xn1isClVar}
\end{equation}
Assume now that columns $n,\ldots j+1$ were mutated, and look at
the sequence $T_{1}$ on the column $j.$ Start at $\left(n,j\right)$
where the exchange relation is shown in Fig. \ref{fig:T1atnj}: 
\begin{eqnarray*}
\varphi_{n,j}\varphi_{n,j}^{1} & = & \varphi_{n,j-1}\varphi_{n,j+1}^{1}+\varphi_{n-1,j-1}\varphi_{21}\\
 & = & x_{n,j-1}\tilde{f}_{n-1,j}^{[1]}+f_{n-1,j-1}f_{21}\\
 & = & x_{n,j}\tilde{f}_{n-1,j-1}^{[1]}
\end{eqnarray*}
\begin{figure}
\begin{centering}
\includegraphics[scale=0.5]{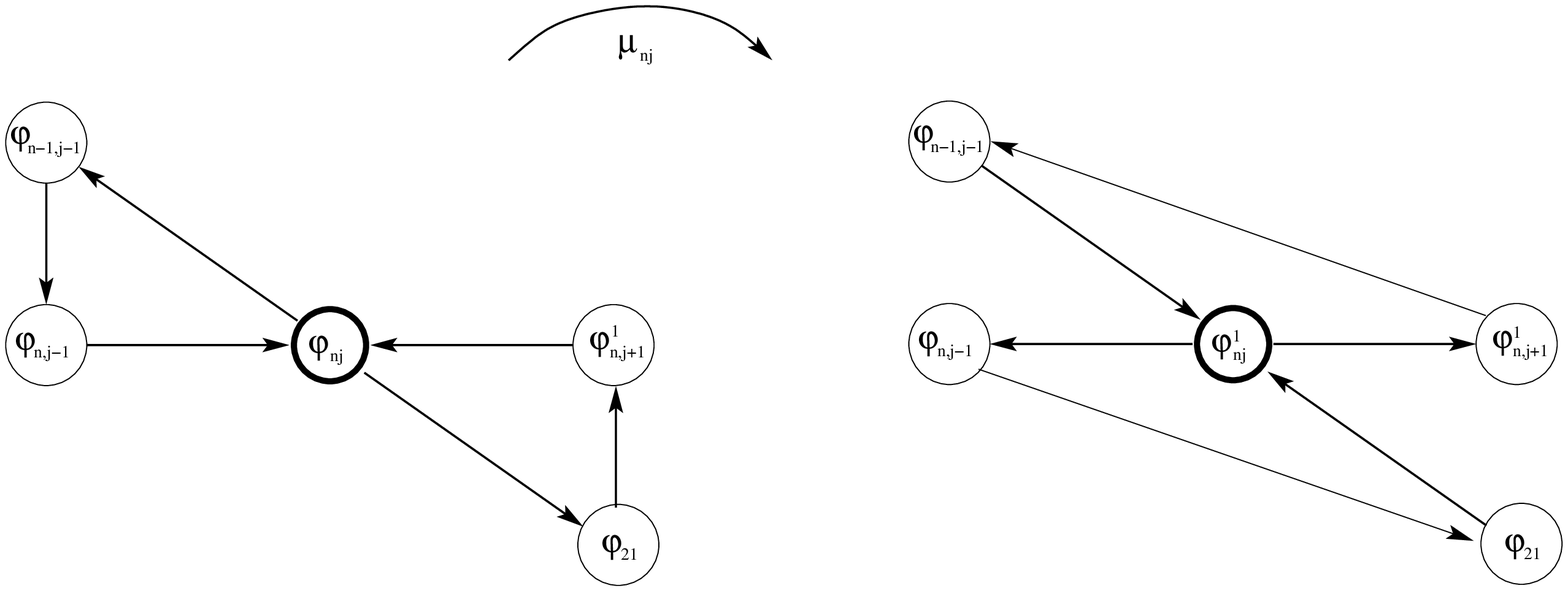} 
\par\end{centering}

\caption{
Mutation $T_{1}$ at vertex $\left(n,j\right)$
}

\label{fig:T1atnj} 
\end{figure}
so 
\[
\varphi_{n,n-1}^{1}=\tilde{f}_{n-1,j-1}^{[1]}.
\]
 The arrows $\left(n-1,j-1\right)\to\left(n,j-1\right)$ and $\left(2,1\right)\to\left(n,j+1\right)$
were removed, and two arrows were added: $\left(n,j-1\right)\to\left(n-1,j-1\right)$
and $\left(n,j-1\right)\to\left(2,1\right)$. The arrows that touch
$\left(n,j\right)$ were inverted. 

We move on, mutating at $\left(k,j\right)$ 
with $k>j$. The exchange relation is now given in Fig. \ref{fig:T1atkgj}
\begin{eqnarray*}
\varphi_{kj}\varphi_{kj}^{1} & = & \varphi_{k,j-1}\varphi_{k,j+1}^{1}+\varphi_{k-1,j-1}\varphi_{k+1,j+1}^{1}\\
 & = & f_{k,j-1}\tilde{f}_{k,j+1}^{[1]}+f_{k-1,j-1}\tilde{f}_{k+1,j+1}^{[1]}\\
 & = & f_{21}\left(f_{k,j-1}f_{k,j+1}^{\left(1\right)}+f_{k-1,j-1}f_{k,j}^{\left(1\right)}\right)\\
 &  & -f_{21}^{\uparrow}\left(f_{k,j-1}f_{k,j+1}^{\left(1\right)\downarrow}+f_{k-1,j-1}f_{k+1,j+1}^{\left(1\right)\downarrow}\right)
\end{eqnarray*}
\begin{figure}
\begin{centering}
\includegraphics[scale=0.5]{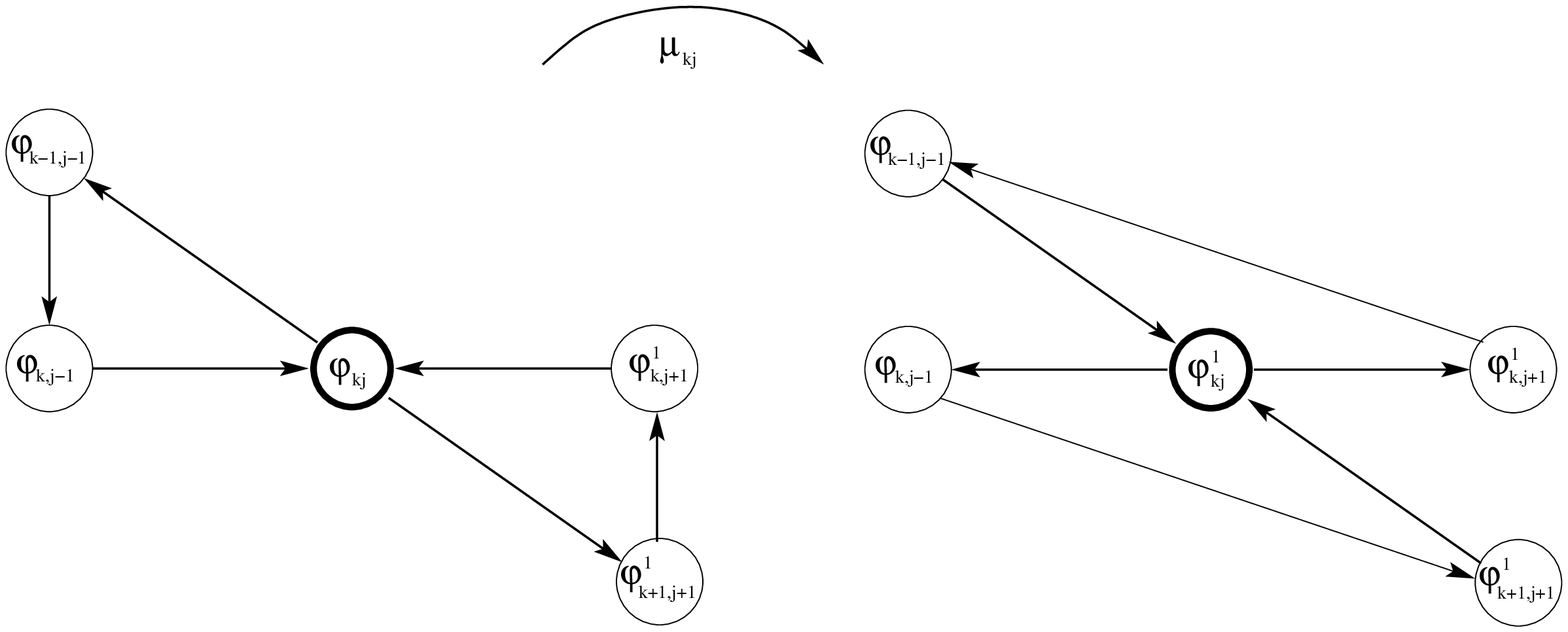} 
\par\end{centering}

\caption{
Mutation $T_{1}$ at vertex $\left(k,j\right)$ with $k>j$
}

\centering{}\label{fig:T1atkgj} 
\end{figure}
 and we can use \eqref{eq:DesJacId} with
$A=X_{\left[k-1\ldots n\right]}^{\left[j-1\ldots\ell\right]}$ (with
$\ell=n-k+j$), so the first parenthesis is just $f_{kj}f_{k-1,j-1}^{\left(1\right)}$
and the second one is $f_{kj}f_{k-1,j-1}^{\downarrow}$, and so 
\[
\varphi_{kj}^{1}=\tilde{f}_{k-1,j-1}^{[1]}.
\]
 After this mutation the arrows $\left(k-1,j-1\right)\to\left(k,j-1\right)$
and $\left(k+1,j+1\right)\to\left(k,j+1\right)$ were removed, and
two arrows were added: $\left(k,j-1\right)\to\left(k+1,j+1\right)$
and $\left(k,j+1\right)\to\left(k-1,j-1\right)$. The arrows that
touch $\left(k,j\right)$ were inverted.

Next look at the mutation on the main diagonal, at $\left(j,j\right)$:
Fig. \ref{fig:T1atjj} describes the exchange relation:
\begin{eqnarray*}
\varphi_{jj}\varphi_{jj}^{1} & = & \varphi_{j-1,j-1}\varphi_{j+1,j+1}^{1}+\varphi_{j,j-1}\varphi_{j-1,j}\\
 & = & f_{j-1,j-1}\tilde{f}_{jj}^{[1]}+f_{j,j-1}\tilde{f}_{j-1,j}^{[1]}\\
 & = & f_{21}\left(f_{j-1,j-1}f_{jj}^{\left(1\right)}+f_{j,j-1}f_{j-1,j}^{\left(1\right)}\right)\\
 &  & -f_{21}^{\uparrow}\left(f_{j-1,j-1}f_{jj}^{\left(1\right)\downarrow}+f_{j,j-1}f_{j-1,j}^{\left(1\right)\downarrow}\right)
\end{eqnarray*}
\begin{figure}
\begin{centering}
\includegraphics[scale=0.5]{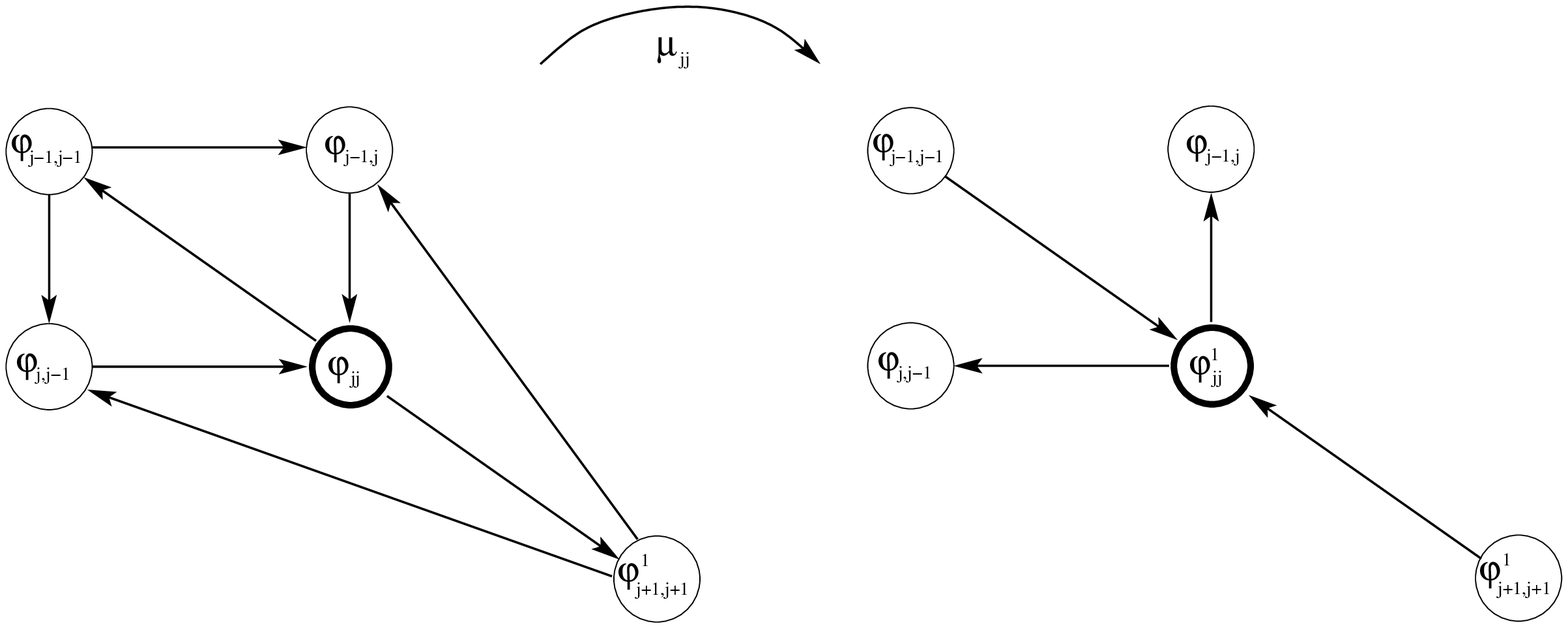} 
\par\end{centering}

\caption{
Mutation $T_{1}$ at vertex $\left(j,j\right)$
}

\label{fig:T1atjj} 
\end{figure}

and we use \eqref{eq:DesJacId} again, now with $A=X_{\left[j-1\ldots n\right]}^{\left[j-1\ldots n\right]}$
to get $f_{jj}f_{j-1,j-1}^{\left(1\right)}$ in the first parenthesis
and $f_{jj}f_{j-1,j-1}^{\left(1\right)\downarrow}$ in the second
one. Thus the exchanged variable is 
\[
\varphi_{jj}^{1}=\tilde{f}_{j-1,j-1}.
\]
 Four arrows were removed: $\left(j-1,j-1\right)\to\left(j,j-1\right),\left(j-1,j-1\right)\to\left(j-1,j\right),\left(j+1,j+1\right)\to\left(\left(j,j-1\right)\right)$
and $\left(j+1,j+1\right)\to\left(j-1,j\right)$. The four arrows
that touch $\left(j,j\right)$ were inverted.

proceeding along column $j$ we mutate now at $\left(k,j\right)$
where $k<j$. The exchange relation here is 
\begin{equation}
\varphi_{kj}\varphi_{kj}^{1}=\varphi_{k-1,j-1}\varphi_{k+1,j+1}^{1}+\varphi_{k-1,j}\varphi_{k+1,j}^{1},\label{eq:er@kj_k<j}
\end{equation}
 and there are two cases here:

First, if $k=j-1$ then \eqref{eq:er@kj_k<j} reads 
\begin{eqnarray*}
\varphi_{kj}\varphi_{kj}^{1} & = & \varphi_{j-2,j-1}\varphi_{j,j+1}^{1}+\varphi_{j-2,j}\varphi_{j,j}^{1}\\
 & = & \tilde{f}_{j-2,j-1}f_{j-1,j}^{\left(1\right)}+f_{j-2,j}\tilde{f}_{j-1,j-1}^{[1]}\\
 & = & f_{21}\left(f_{j-2,j-1}f_{j-1,j}^{\left(1\right)}+f_{j-2,j}f_{j-1,j-1}^{\left(1\right)}\right)\\
 &  & -f_{21}^{\uparrow}\left(f_{j-2,j-1}f_{j-1,j}^{\left(1\right)\downarrow}+f_{j-2,j}f_{j-1,j-1}^{\left(1\right)\downarrow}\right)
\end{eqnarray*}
\begin{figure}
\begin{centering}
\includegraphics[scale=0.5]{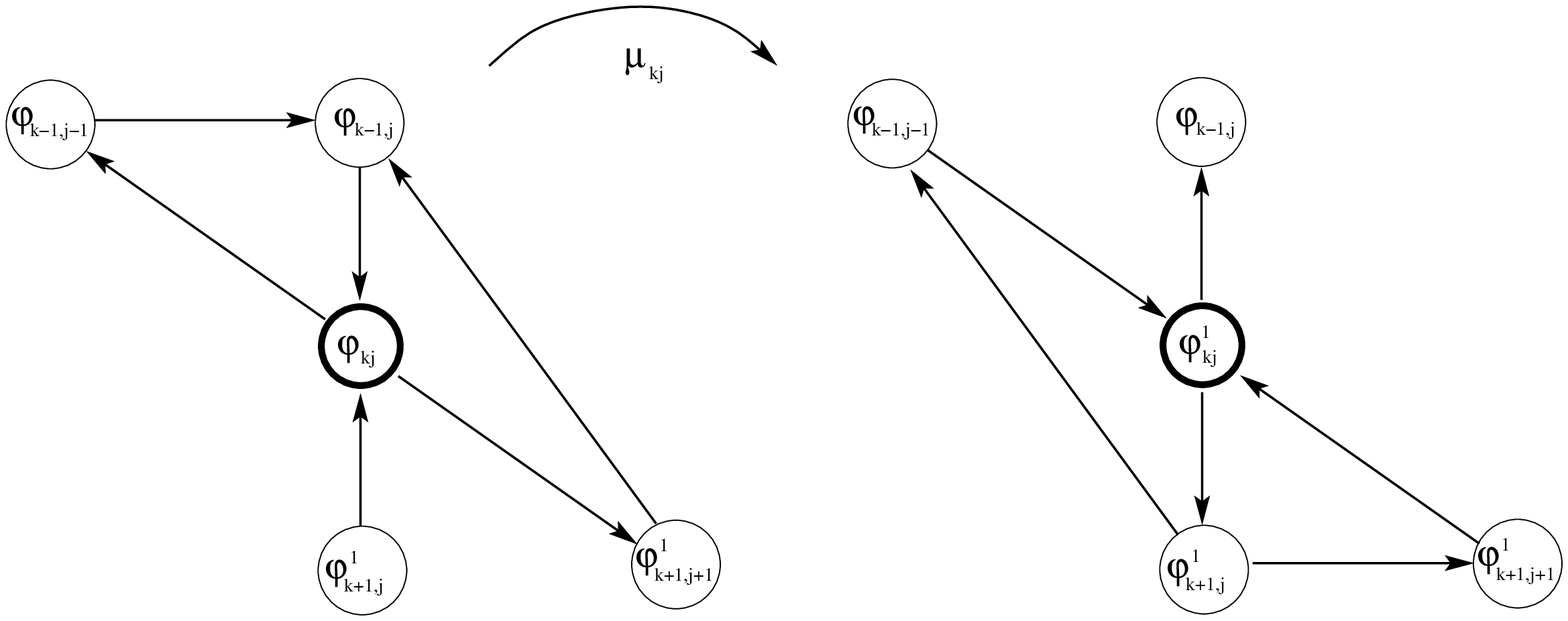} 
\par\end{centering}

\caption{
Mutation $T_{1}$ at vertex $\left(k,j\right)$ with $k<j$
}

\label{fig:T1atklj-1} 
\end{figure}

and using \eqref{eq:DesJacId} with $A=X_{\left[j-2\ldots n-1\right]}^{\left[j-1\ldots n\right]}$
yields $f_{j-1,j}f_{j-2,j-1}^{\left(1\right)}$ in the first parenthesis,
and with $A=X_{\left[j-2\ldots n-2,n\right]}^{\left[j-1\ldots n\right]}$
gives $f_{j-1,j}f_{j-2,j-1}^{\left(1\right)}$ in the second one.
Recall that $\varphi_{j-1,j}=\tilde{f}_{j-1,j}$ and therefore 
\[
\varphi_{j-1,j}^{1}=f_{j-2,j-1}^{\left(1\right)}.
\]
The second case is when $k<j-1$ and then \eqref{eq:er@kj_k<j} reads
\[
\varphi_{kj}\varphi_{kj}^{1}=f_{k-1,j-1}f_{k,j}^{\left(1\right)}+f_{k-1,j}f_{k,j-1}^{\left(1\right)}
\]
 and with \eqref{eq:DesJacId} on the matrix $A=X_{\left[k-1\ldots\ell\right]}^{\left[j-1\ldots n\right]}$
(here $\ell=n+k-j$) it comes down to 
\[
\varphi_{kj}^{1}=f_{k-1,j-1}^{\left(1\right)}.
\]

In both cases the quiver changes are: the arrows $\left(k-1,j-1\right)\to\left(k-1,j\right)$
and $\left(k+1,j+1\right)\to\left(k-1,j\right)$ were removed and
$\left(k+1,j\right)\to\left(k-1,j-1\right)$ and $\left(k+1,j\right)\to\left(k+1,j+1\right)$
were added. \end{proof}
\begin{prop}
\label{prop:SeedS2}In the seed $\Sigma_{2}$, 

1. for all $i,j\in\left\{ 3,\ldots,n\right\} $
the cluster variables take the form 
\begin{eqnarray*}
\varphi_{ij}^{2} & = & f_{i-2,j-2}^{\left(2\right)};
\end{eqnarray*}

2. after freezing all the vertices $\left(3,j\right)$ and $\left(i,3\right)$,
all vertices $\left(2,j\right)$ and $\left(i,2\right)$ become isolated
and can be ignored. The subquiver on vertices $i,j\in\left\{ 3,\ldots,n\right\} $
is isomorphic to the standard quiver $Q_{std}(n-2)$ on $SL_{n-2}$.\end{prop}
\begin{proof}
Following the pattern of the proof of Proposition \ref{prop:SeedS1},
we can look at the mutations step by step. All exchange relations
can be resolved using the Desnanot--Jacobi identity \eqref{eq:DesJacId}.
If all the functions in the exchange relation are proper minors of
the matrix $X$, this is pretty much straightforward. If the exchange
relation involves a function of the form $\tilde{f}_{ij}$, it is
still not too hard: the structure of the quiver assures there must
be two such functions (this can be easily proved by induction), each
of which has the form $g\cdot f_{21}+g^{\downarrow}f_{21}^{\uparrow}$.
The arrows that connect the corresponding two vertices to the vertex
associated with the exchanged variable point in opposite directions
(one towards this vertex and the other away from it. This is also
not hard to see). Therefore, the exchange relation can be broken into
two parts: the first has only determinants of dense submatrices of
$X$ multiplied by $f_{21}$. The second one has determinants of same
submatrices with just one row replaced by another (recall that if
$g=\det X_{[i,k]}^{[j,\ell]}$ then $g^{\downarrow}=\det X_{[i,...,k-1,k+1]}^{[j,\ell]}$).
Using \eqref{eq:DesJacId} on each part separately yields the result.\end{proof}
\begin{prop}
\label{prop:Seed>2}For $m\ge3$, 

1. The subquiver of $Q^{m}$ of rows $m+1,\ldots,n$ and columns $m+1,\ldots,n$
is isomorphic to the standard quiver $Q_{std}(n-m)$ (on $SL_{n-m}$).

2. The functions $\varphi_{ij}^{m}\in\mathcal{B}_{m}$ with $j>i$
are 
\[
\varphi_{ij}^{m}=f_{i-m,j-m}^{\left(m\right)}.
\]
\end{prop}
\begin{proof}
By induction, based on Propositions \ref{prop:SeedS1} and \ref{prop:SeedS2}.
All the relevant exchange relations look like the standard ones.\end{proof}
\begin{cor}
If $i<j<n$, then $x_{ij}$ is a cluster variable.\label{cor:xijinAforalljgi}\end{cor}
\begin{proof}
Set $m=n-j,$ and then $\varphi_{i+m,n}^{m}=f_{i,n-m}^{\left(m\right)}=x_{ij}$
is a cluster variable.\end{proof}
\begin{lem}
For every $(i,j)\in[n-1]\times[n-1]$ the function $x_{ij}$ belongs
to the upper cluster algebra $\mathcal{\overline{A}}_{1\mapsto n-1}$
.\label{lem:xijinA1n-1}\end{lem}
\begin{proof}
First, by Corollary \ref{cor:xijinAforalljgi}, $x_{ij}\in\mathcal{\overline{A}}_{1\mapsto n-1}$
for all $j>i$. 

To see that $x_{n-1,n-1}$ belongs to $\mathcal{\overline{A}}_{1\mapsto n-1}$
note that 
\[
\varphi_{n-1,n-1}=\left|\begin{array}{cc}
x_{n-1,n-1} & x_{n-1,n}\\
x_{n,x-1} & x_{nn}
\end{array}\right|
\]
 and we have 
\[
\varphi_{nn}^{1}=x_{n-1,n-1}f_{21}-x_{n,n-1}f_{21}^{\uparrow},
\]
 so 
\[
x_{n-1,n-1}=\frac{\varphi_{n-1,n-1}+x_{n-1,n}x_{n,n-1}}{x_{nn}}=\frac{\varphi_{nn}^{1}+x_{n,n-1}f_{21}^{\uparrow}}{\varphi_{21}}.
\]
So according to Lemma \ref{lem:InAGSV}, $x_{n-1,n-1}\in\overline{\mathcal{A}}$.

If $i=n-1$ and $j<i$ then 
\[
\varphi_{i,j}=\left|\begin{array}{cc}
x_{n-1,j} & x_{n-1,j+1}\\
x_{nj} & x_{n,j+1}
\end{array}\right|
\]
 and 
\[
\varphi_{n,j+1}^{1}=x_{n-1,j}f_{21}-x_{nj}f_{21}^{\uparrow}.
\]
 So 
\[
x_{n-1,j}=\frac{\varphi_{n-1,j}+x_{nj}x_{n-1,j+1}}{x_{n,j+1}}=\frac{\varphi_{n,j+1}^{1}+x_{nj}f_{21}^{\uparrow}}{f_{21}}.
\]
Inductively assuming $x_{n-1,j+1}\in\overline{\mathcal{A}}$, this
satisfy the conditions of Lemma \ref{lem:InAGSV}, and therefore $x_{ij}\in\mathcal{\overline{A}}_{1\mapsto n-1}$. 

If $i<n-1$ we use induction: Let $D_{ij}$ be the set of all $x_{k\ell}$
with $k\ge i$ and $\ell\ge j$ without $x_{ij}$. For a pair $(i,j)$
assume that $D_{ij}\subseteq\mathcal{\overline{A}}_{1\mapsto n-1}$
(that is, all $x_{k\ell}$ with $k\ge i$ and $\ell\ge j$ are in
$\mathcal{\overline{A}}_{1\mapsto n-1}$, except maybe $x_{ij}$ itself).
We have 
\[
\varphi_{ij}=\det X_{[i,n]}^{[j,\mu]}=x_{ij}\cdot f_{i+1,j+1}-p_{1}
\]
 where $p_{1}$ is a polynomial in the variables $x_{k\ell}\in D_{ij}$
and therefore $p_{1}\in\mathcal{\overline{A}}_{1\mapsto n-1}$. Similarly,
\[
\varphi_{ij}^{2}=\det X_{[i,n-2]}^{[j,\mu-2]}=x_{ij}f_{i+1,j+1}^{(2)}-p_{2}
\]
 with $p_{2}\in\mathcal{\overline{A}}_{1\mapsto n-1}$ again. Hence,
\[
x_{ij}=\frac{\varphi_{ij}+p_{1}}{f_{i+1,j+1}}=\frac{\varphi_{ij}^{2}+p_{2}}{f_{i+1,j+1}^{(2)}}
\]
 and so by Lemma \ref{lem:InAGSV} we get $x_{ij}\in\mathcal{\overline{A}}_{1\mapsto n-1}$.
\end{proof}

\section{The Toric action\label{sec:The-Toric-action}}

Proving Theorem \ref{thm:The-global-toric} is equivalent
to proving the following (see \cite{gekhtman2012cluster}):
\begin{enumerate}
\item \label{enu:TAeqCon1}for any $H_{1},H_{2}\in\mathcal{H}_{T}$ and
any $X\in SL_{n}$, 
\[
y_{i}\left(H_{1}XH_{2}\right)=H_{1}^{\eta_{i}}H_{2}^{\zeta_{i}}y_{i}\left(X\right)
\]
 for some weights $\eta_{i},\zeta_{i}\in\mathfrak{h}_{T}^{*}\ \left(i\in[n+m]\right)$;
\item \label{enu:TAeqCon2}$\spn\left\{ \eta_{i}\right\} _{i=1}^{\dim\mathcal{G}}=\spn\left\{ \zeta_{i}\right\} _{i=1}^{\dim\mathcal{G}}=\mathfrak{h}_{T}^{*}$
;
\item \label{enu:TAeqCon3}for every $i\in\left[\dim SL_{n}-2k_{T}\right]$,
\[
\sum_{j=1}^{\dim SL_{n}}b_{ij}\eta_{j}=\sum_{j=1}^{\dim SL_{n}}b_{ij}\zeta_{j}=0.
\]

\end{enumerate}
For a seed $\left(\tilde{\mathbf{x}},\tilde{B}\right)$ in $\mathcal{C}_{T}$,
and $y_{i}=\varphi\left(x_{i}\right)$ for $i\in[n+m]$.
\begin{prop}
For any BD triple $T=(\{\alpha\},\{\beta\},\alpha\mapsto\beta)$ on
$SL_{n}$ statements 1,2 and 3 above hold true.\end{prop}
\begin{proof}
Recall that 
\begin{equation}
\mathfrak{h}_{T}=\left\{ h\in\mathfrak{h}:\alpha(h)=\beta(h)\right\} .\label{eq:htDef}
\end{equation}
 Therefore, taking the basis $\left\{ h_{i}=e_{ii}-e_{i+1,i+1}\right\} $
we can parametrize $\mathfrak{h}_{T}$ with linear combinations $h=\sum c_{i}h_{i}$
with coefficient vectors $\left(c_{1},\ldots,c_{n-1}\right)$ subject
to a restriction derived from \eqref{eq:htDef}. To understand this
restriction look at the Cartan matrix $C$ of $\mathfrak{sl}_{n}$
\begin{eqnarray*}
C_{ij} & = & \begin{cases}
2 & i=j\\
-1 & \left|i-j\right|=1\\
0 & \text{otherwise, }
\end{cases}
\end{eqnarray*}
let $[C]_{i}$ denote the $i$-th row of $C$, and let $A$ be the
one row matrix $A=[C]_{\alpha}-[C]_{\beta}$. Then the coefficient
vector must be in the null space of $A$ (to satisfy $\alpha(h)=\beta(h)$).
Since $C$ is a symmetric $(n-1)\times(n-1)$ matrix, $A^{T}=C\left(e_{\alpha}-e_{\beta}\right)$
and clearly its null space is $n-2$ dimensional.

Assume $y_{m}$ is a function of the form $f_{ij}=\det X_{[i,k]}^{[j,\ell]}$.
Take a diagonal matrix $H_{1}\in\mathcal{H}_{T}$ where $H_{1}=\exp h$
for some $h=\diag(d_{1},\ldots,d_{n})$. Now set $\eta_{m}(h)=d_{i}+\ldots+d_{k}$.
It is easy to verify that 
\begin{equation}
y_{m}\left(H_{1}X\right)=H_{1}^{\eta_{m}}y_{m}\left(X\right).\label{eq:RowWght}
\end{equation}
 If $y_{m}$ is a function of the form
\[
\theta_{j}=\det\left[\begin{array}{cccccc}
x_{ij} & \cdots & \cdots & x_{i,\alpha+1} & 0 & \cdots\\
 & \ddots &  & \vdots & \vdots\\
x_{n1} & \cdots & x_{n\alpha} & x_{n,\alpha+1} & 0 & \cdots\\
\cdots & 0 & x_{1\beta} & x_{1,\beta+1} & \cdots & x_{1n}\\
 & \vdots & \vdots & \vdots & \ddots & \vdots\\
 &  & \vdots & x_{\beta j} & \cdots & x_{\beta n}\\
 & 0 & x_{n-\beta} & x_{n-\beta,j} & \cdots & x_{n-\beta,n}
\end{array}\right]
\]
 it is not hard to see that setting $\eta_{m}(h)=d_{1}+\cdots+d_{n-\beta}+d_{i}+\ldots+d_{n}$
yields \eqref{eq:RowWght} again.

The last case is when $y_{m}$ is of type $\psi$ - ,
\[
\psi_{i}=\det\left[\begin{array}{cccccc}
x_{ij} & \cdots & x_{in} & 0 & \cdots & 0\\
\vdots & \ddots & \vdots & \vdots &  & \vdots\\
x_{\beta j} & \cdots & x_{\beta n} & x_{\alpha1} & \cdots & x_{\alpha,n-\alpha}\\
x_{\beta+1,j} & \cdots & x_{\beta+1,n} & x_{\alpha+1,1} &  & \vdots\\
0 & \cdots & 0 & \vdots & \ddots & \vdots\\
\vdots &  & \vdots & x_{n1} & \cdots & x_{n,n-\alpha}
\end{array}\right]
\]
with $j=n+i-\beta$. We can write $\psi_{i}=f_{ij}f_{\alpha+1,1}-f_{ij}^{\downarrow}f_{\alpha+1,1}^{\uparrow}$,
or using determinants of submatrices 
\[
\psi_{i}=\det X_{[i,\beta]}^{[j,n]}\det X_{[\alpha+1,n]}^{[1,n-\alpha]}-\det X_{[1,\ldots,\beta-1,\beta+1]}^{[j,n]}\det X_{[\alpha,\alpha+2,\ldots,n]}^{[1,n-\alpha]}.
\]
 Therefore, \eqref{eq:RowWght} holds if 
\begin{equation}
d_{\beta}+d_{\alpha+1}=d_{\beta+1}+d_{\alpha},\label{eq:SumWghtsab}
\end{equation}
 i.e., the sum of weights on rows $\beta,\alpha+1$ is equal to the
sum of those on rows $\beta+1,\alpha$. In this case the weight is
$\eta_{m}(h)=d_{i}+\cdots+d_{\beta}+d_{\alpha+1}+\cdots d_{n}$. This
is equivalent to the condition that $h=\diag(d_{1},\ldots,d_{n})$
is in the null space of the row matrix $B=e_{\alpha}-e_{\alpha+1}-e_{\beta}+e_{\beta+1}$
(here $e_{k}$ is a row vector with $1$ in the $k$-th entry and
$0$ elsewhere). Let $T$ be the transformation matrix from the basis
$\left\{ h_{i}\right\} _{i=1}^{n-1}$ to the standard one (so the
$j$-th column of $T$ is $e_{j}-e_{j+1}$), and let $v^{T}=\left(c_{1},\ldots c_{n-1}\right)$
be the coefficient vector of $h$ in the basis $\left\{ h_{i}\right\} _{i=1}^{n-1}$
(so that $h=Tv$). It is not hard to verify that $BT=[C]_{\alpha}-[C]_{\beta}$,
hence 
\[
Bh=BTv=Av=0,
\]
because $v$ is in the null space of $A$ by its definition. This
establishes \eqref{eq:SumWghtsab}, and therefore \eqref{eq:RowWght}
holds in this case as well. 

Repeating the above with right multiplication by $H_{2}$, now with
weights $\zeta_{m}$, yields statement \ref{enu:TAeqCon1}.

Proving statement \ref{enu:TAeqCon2} is easy: as described above,
the weight associated with the function $f_{ii}$ is $\overline{\eta}_{i}(h)=d_{i}+\cdots+d_{n}$.
Therefore the set $\{\overline{\eta}_{2},\ldots,\overline{\eta}_{n}\}$
spans $\mathfrak{h}^{*}$, because this is a set of $n-1$ linearly
independent vectors in an $n-1$ dimensional space. Since $\mathfrak{h}_{T}^{*}$
is a subspace of $\mathfrak{h}^{*}$ (because $\mathfrak{h}_{T}$
is a subspace of $\mathfrak{h}$), this set also spans $\mathfrak{h}_{T}^{*}.$
The same holds for the weights $\zeta_{i}$.

To prove statement \ref{enu:TAeqCon3} we rephrase it as follows:
for every mutable vertex $v=(i,j)$ of the quiver, the sum of weights
over the neighbors with arrows pointing towards $v$ is equal to the
sum of weights over the neighbors with arrows pointing out of $v$.
This is true because $b_{ij}\in\{0,1,-1\}$. It will be proved for
the weights $\eta_{m}$ corresponding to the left multiplication $y_{m}\left(HX\right)=H^{\eta_{m}}y_{m}\left(X\right)$
but symmetric arguments will hold for $\zeta_{m}$ with right multiplication
to show $y_{m}\left(XH\right)=H^{\zeta_{m}}y_{m}\left(X\right)$.

As was explained above, for $h=\diag(d_{1},\ldots,d_{n})$ the weights
$\eta_{i}$ can be defined as 
\[
\eta_{m}(h)=\begin{cases}
d_{i}+\cdots+d_{k} & \text{ if }y_{m}=f_{ij}=\det X_{[i,k]}^{[j,\ell]}\\
d_{i}+\cdots+d_{n}+d_{1}+\cdots+d_{n-\beta} & \text{ if }y_{m}=\theta_{j}=f_{ij}f_{1,\beta+1}-f_{ij}^{\rightarrow}f_{1,\beta+1}^{\leftarrow}\\
d_{i}+\cdots+d_{\beta}+d_{\alpha+1}+\cdots+d_{n} & \text{ if }y_{m}=\psi_{i}=f_{ij}f_{\alpha+1,1}-f_{ij}^{\downarrow}f_{\alpha+1,1}^{\uparrow}.
\end{cases}
\]
Assume $y_{m}$ and $y_{m'}$ are two cluster variables that correspond
to two vertices on the same diagonal of the quiver, that is, $y_{m}$
corresponds to the vertex $(i,j)$ and $y_{m'}$ corresponds to $(i+k,j+k).$
Then 
\begin{equation}
\eta_{m}(h)=\eta_{m'}(h)+d_{i}+\cdots+d_{k-1}.\label{eq:WtsonSmDiag}
\end{equation}
 With this fact in mind, look at the sum of weights over all the vertices
adjacent to a mutable vertex $(i,j)$ and consider the following cases:

1. If $1<i,j<n$ then there are three arrows pointing to $(i,j)$
from the vertices $A=(i-1,j),\ B=(i,j-1)$ and $C=(i+1,j+1)$ and
three arrows from $(i,j)$ to the vertices $D=(i,j+1),\ E=(i+1,j)$
and $F=(i-1,j-1)$ (see Fig. \ref{fig:Nbrsij}). According to \eqref{eq:WtsonSmDiag},
\begin{eqnarray*}
\eta_{A} & = & \eta_{D}+d_{i-1}\\
\eta_{B} & = & \eta_{E}+d_{i}\\
\eta_{C} & = & \eta_{F}-d_{i}-d_{i-1}
\end{eqnarray*}
and therefore $\eta_{A}+\eta_{B}+\eta_{C}=\eta_{D}+\eta_{E}+\eta_{F}$.

\begin{figure}
\begin{centering}
\includegraphics[scale=0.45]{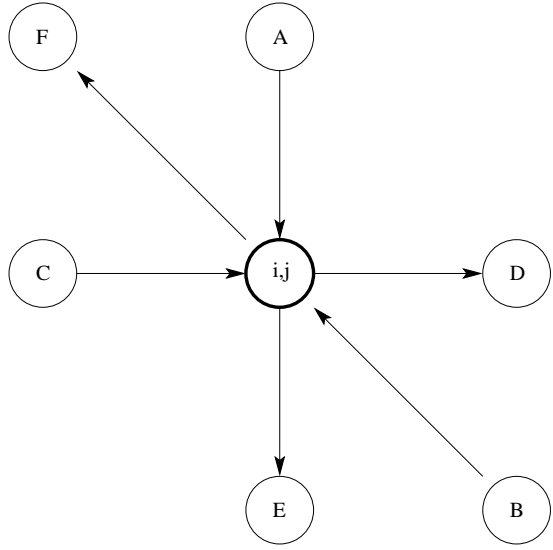} 
\par\end{centering}

\caption{The neighbors of $(i,j)$ when $1<i,j<n$}

\centering{}\label{fig:Nbrsij} 
\end{figure}

2. If $i=j=n$ then $(i,j)$ has only three neighbors (\footnote{If $\beta=n-1$ then $D=(\alpha+1,1)$ is also a neighbor and it is
easy to verify that the result holds.}Assuming $\beta\neq n-1$) as shown in Fig. \ref{fig:Nbrnn}: $A=\left(n-1,n\right)$
and $B=(n,n-1)$ with arrows to $(n,n)$ and $C=(n-1,n-1)$ with an
arrow pointing to it from $(n,n)$. In this case it is clear that
\begin{eqnarray*}
\eta_{A} & = & d_{n-1}\\
\eta_{B} & = & d_{n}\\
\eta_{C} & = & d_{n-1}+d_{n}
\end{eqnarray*}
 and so $\eta_{A}+\eta_{B}=\eta_{C}$. 

\begin{figure}
\begin{centering}
\includegraphics[scale=0.45]{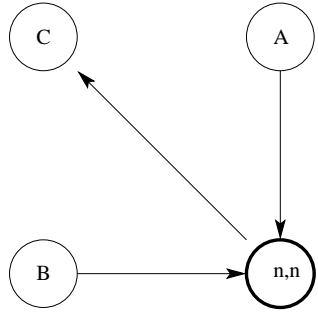} 
\par\end{centering}

\caption{The neighbors of $(n,n)$}
\label{fig:Nbrnn} 
\end{figure}

3. If $i=n$ and $j\notin\{\alpha,\alpha+1,n\}$, then there are four
neighboring vertices: $A=(n,j-1)$ and $B=(n-1,j)$ with arrows pointing
to $(n,j)$, and $C=(n,j+1)$ and $D=(n-1,j-1)$ with arrows pointing
to them. Here $y_{A}=x_{n,j-1}$ and $y_{D}=\left|\begin{array}{cc}
x_{n-1,j-1} & x_{n-1,j}\\
x_{n,j-1} & x_{nj}
\end{array}\right|$ so $\eta_{A}=d_{n}$ and $\eta_{D}=d_{n-1}+d_{n}.$ According to
\eqref{eq:WtsonSmDiag} $\eta_{B}=\eta_{C}+d_{n-1}$, so again, $\eta_{A}+\eta_{B}=\eta_{C}+\eta_{D}$.

4. The vertex $(n,\alpha)$ has three neighbors with arrows pointing
at it: $A=(n-1,\alpha),\ B=(1,\beta+1)$ and $C=(n,\alpha-1)$, and
two neighbors with arrows from $(n,\alpha)$ to them: $D=(n,\alpha+1)$
and $E=(n-1,\alpha-1)$ (see Fig. \ref{fig:Nbrna}). We have
\begin{eqnarray*}
\eta_{A} & = & \eta_{D}+d_{n-1}\\
\eta_{B} & = & \eta_{E}-d_{n-1}-d_{n}\\
\eta_{C} & = & d_{n}
\end{eqnarray*}
 and $\eta_{A}+\eta_{B}+\eta_{C}=\eta_{D}+\eta_{E}$.

\begin{figure}
\begin{centering}
\includegraphics[scale=0.45]{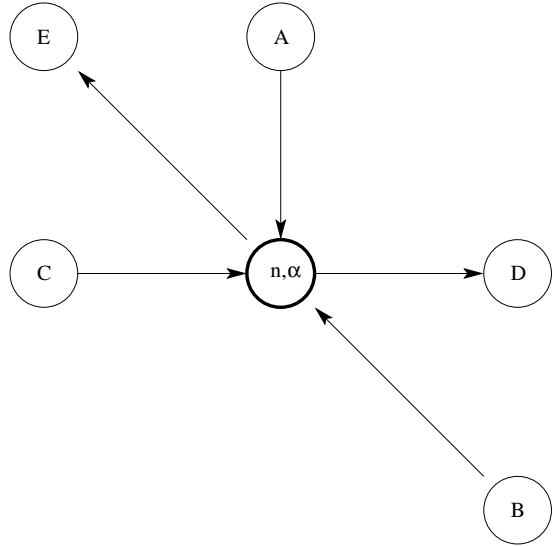} 
\par\end{centering}

\caption{The neighbors of $(n,\alpha)$}

\centering{}\label{fig:Nbrna} 
\end{figure}

5. The vertex $(n,\alpha+1)$ has two neighbors with arrows pointing
at it: $A=(n-1,\alpha+1)$ and $B=(n,\alpha)$. There are three neighbors
with arrows from $(n,\alpha+1)$ to them: $C=(n,\alpha+2)$, $D=(1,\beta+1)$
and $E=(n-1,\alpha)$. Fig. \ref{fig:Nbrna1} shows $(n,\alpha+1)$
and its neighbors. So here 

\begin{eqnarray*}
\eta_{A} & = & \eta_{C}+d_{n-1}\\
\eta_{B} & = & \eta_{D}+d_{n}\\
\eta_{E} & = & d_{n-1}+d_{n}
\end{eqnarray*}
and $\eta_{A}+\eta_{B}=\eta_{C}+\eta_{D}+\eta_{E}$.

\begin{figure}
\begin{centering}
\includegraphics[scale=0.45]{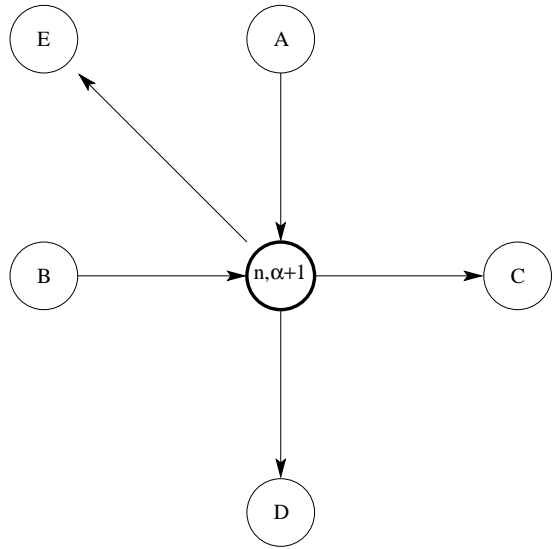} 
\par\end{centering}

\caption{The neighbors of $(n,\alpha+1)$}

\centering{}\label{fig:Nbrna1} 
\end{figure}

6. The vertex $(i,n)$ with $i\notin\{\beta,\beta+1,n\}$. There are
four neighbors as Fig. \ref{fig:Nbrin} shows: $A=(i-1,n)$ and $B=(i,n-1)$
with arrows pointing to $(i,n)$, and $C=(i+1,n)$ and $D=(i-1,n-1)$
with arrows pointing to them. Here $y_{A}=x_{i-1,n}$ and $y_{D}=\left|\begin{array}{cc}
x_{i-1,n-1} & x_{i-1,n}\\
x_{i,n-1} & x_{in}
\end{array}\right|$ so $\eta_{A}=d_{i-1}$ and $\eta_{D}=d_{i-1}+d_{i}.$ According to
\eqref{eq:WtsonSmDiag} $\eta_{B}=\eta_{C}+d_{n-1}$, so again, $\eta_{A}+\eta_{B}=\eta_{C}+\eta_{D}$.

\begin{figure}
\begin{centering}
\includegraphics[scale=0.45]{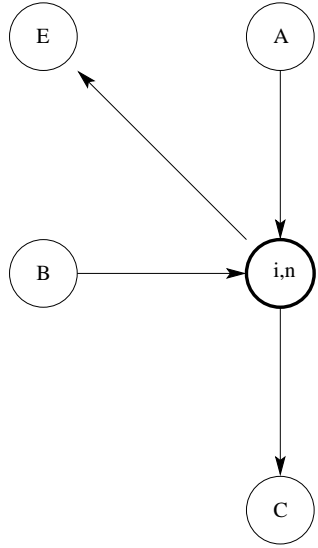} 
\par\end{centering}

\caption{The neighbors of $(i,n)$}

\label{fig:Nbrin} 
\end{figure}

7. The vertex $(\beta,n)$ has three neighbors with arrows pointing
at it: $A=(\beta-1,n),\ B=(\alpha+1,1)$ and $C=(\beta,n-1)$, and
two neighbors with arrows from $(\beta,n)$ to them: $D=(\beta+1,n)$
and $E=(\beta-1,n-1)$ (see Fig. \ref{fig:Nbrbn}). We have
\begin{eqnarray*}
\eta_{A} & = & d_{\beta-1}\\
\eta_{B} & = & \eta_{E}-d_{\beta-1}-d_{\beta}\\
\eta_{C} & = & \eta_{D}+d_{\beta}
\end{eqnarray*}
 and therefore $\eta_{A}+\eta_{B}+\eta_{C}=\eta_{D}+\eta_{E}.$

\begin{figure}
\begin{centering}
\includegraphics[scale=0.4]{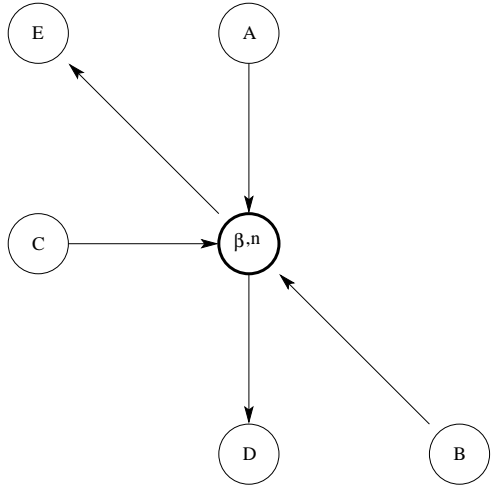} 
\par\end{centering}

\caption{The neighbors of $(\beta,n)$}

\label{fig:Nbrbn} 
\end{figure}

8. The vertex $(\beta+1,n)$ has two neighbors with arrows pointing
at it: $A=(\beta,n)$ and $B=(\beta+1,n-1)$. There are three neighbors
with arrows from $(\beta+1,n)$ to them: $C=(\alpha+1,1)$, $D=(\beta+2,n)$
and $E=(\beta,n-1)$ (see Fig \ref{fig:Nbrb1n}). So here 

\begin{eqnarray*}
\eta_{A} & = & \eta_{C}+d_{\beta}\\
\eta_{B} & = & \eta_{D}+d_{\beta+1}\\
\eta_{E} & = & d_{\beta}+d_{\beta+1}
\end{eqnarray*}
and $\eta_{A}+\eta_{B}=\eta_{C}+\eta_{D}+\eta_{E}$.

\begin{figure}
\begin{centering}
\includegraphics[scale=0.4]{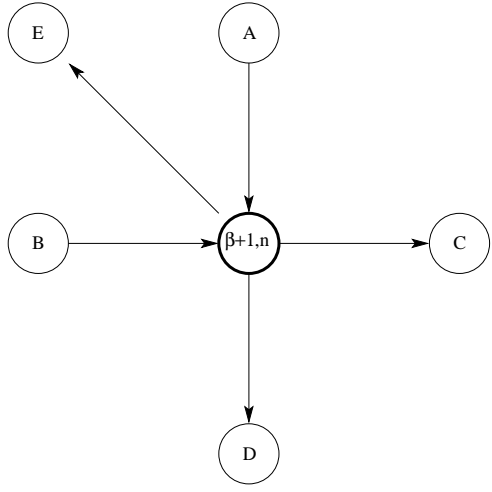} 
\par\end{centering}

\caption{The neighbors of $(\beta+1,n)$}

\label{fig:Nbrb1n} 
\end{figure}

9. The vertex $\left(1,\beta+1\right)$ has three neighbors with arrows
pointing at it: $A=(n,\alpha+1),$ $B=(2,\beta+2)$ and $C=(1,\beta)$.
There are two neighbors with arrows from $(1,\beta+1)$ to them: $D=(2,\beta+1)$
and $E=(n,\alpha)$ (see Fig. \ref{fig:Nb1b1}). So here 

\begin{eqnarray*}
\eta_{A} & = & d_{n}\\
\eta_{B} & = & \eta_{E}-d_{n}-d_{1}\\
\eta_{C} & = & \eta_{D}+d_{1}
\end{eqnarray*}
and again we have $\eta_{A}+\eta_{B}+\eta_{C}=\eta_{D}+\eta_{E}$.

\begin{figure}
\begin{centering}
\includegraphics[scale=0.4]{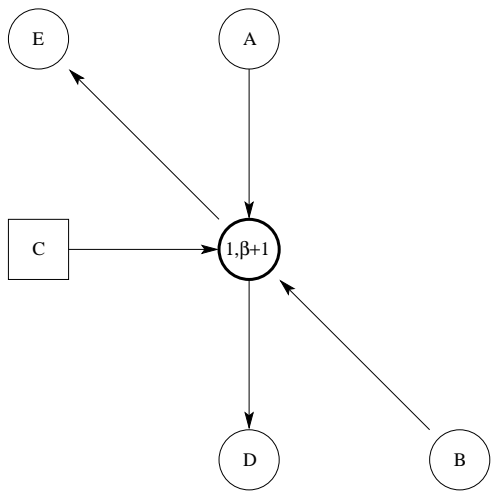} 
\par\end{centering}

\caption{The neighbors of $(1,\beta+1)$}

\label{fig:Nb1b1} 
\end{figure}

10. Last is the vertex $\left(\alpha+1,1\right)$, with three neighbors
with arrows pointing at it: $A=(1,\alpha),$ $B=(\alpha+2,2)$ and
$C=(\beta+1,n)$ and two neighbors with arrows from $(\alpha+1,1)$
to them: $D=(\alpha+1,2)$ and $E=(\beta,n)$ (see Fig. \ref{fig:Nbra11}).
So now

\begin{eqnarray*}
\eta_{A} & = & \eta_{D}+d_{\alpha}\\
\eta_{B} & = & \eta_{E}-d_{\alpha+1}-d_{\beta}\\
\eta_{C} & = & d_{\beta+1}
\end{eqnarray*}
and with \eqref{eq:SumWghtsab}, the result is $\eta_{A}+\eta_{B}+\eta_{C}=\eta_{D}+\eta_{E}$.

\begin{figure}
\begin{centering}
\includegraphics[scale=0.4]{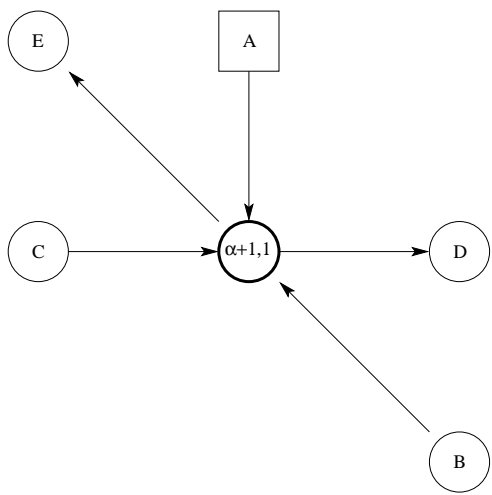} 
\par\end{centering}

\caption{The neighbors of $(\alpha+1,1)$}

\label{fig:Nbra11} 
\end{figure}

\end{proof}

\section*{Acknowledgments}

The author was supported by ISF grant \#162/12. The author thanks
Michael Gekhtman for his helping comments and answers. Special thanks to
Alek Vainshtein for his support and encouragement, as well as his
mathematical, technical and editorial advices.

\newpage{}

\bibliographystyle{abbrv}
\bibliography{Part2L04.bib}

\def\cprime{$'$}
\begin{thebibliography}{10}

\bibitem{BDsolCYBE}
A.~A. Belavin and V.~G. Drinfel{\cprime}d.
\newblock Solutions of the classical {Y}ang-{B}axter equation for simple {L}ie
  algebras.
\newblock {\em Funktsional. Anal. i Prilozhen.}, 16(3):1--29, 96, 1982.

\bibitem{BFZ}
A.~Berenstein, S.~Fomin, and A.~Zelevinsky.
\newblock Cluster algebras. {III}. {U}pper bounds and double {B}ruhat cells.
\newblock {\em Duke Math. J.}, 126(1):1--52, 2005.

\bibitem{Bressoud}
D.~M. Bressoud.
\newblock {\em Proofs and confirmations}.
\newblock MAA Spectrum. Mathematical Association of America, Washington, DC;
  Cambridge University Press, Cambridge, 1999.
\newblock The story of the alternating sign matrix conjecture.

\bibitem{chriprsly}
V.~Chari and A.~Pressley.
\newblock {\em A guide to quantum groups}.
\newblock Cambridge University Press, Cambridge, 1994.

\bibitem{eisner2014SL5}
I.~Eisner.
\newblock Exotic cluster structures on ${SL}_5$.
\newblock {\em Journal of Physics A: Mathematical and Theoretical},
  47(47):474002, 2014.

\bibitem{eisner2015part1}
I.~Eisner.
\newblock Exotic cluster structures on ${SL}_n$ with {B}elavin--{D}rinfeld data
  of minimal size, {I}. the structure.
\newblock {\em arXiv preprint arXiv:1412.5352}, 2014.

\bibitem{FZ1}
S.~Fomin and A.~Zelevinsky.
\newblock Cluster algebras. {I}. {F}oundations.
\newblock {\em J. Amer. Math. Soc.}, 15(2):497--529 (electronic), 2002.

\bibitem{FZ2}
S.~Fomin and A.~Zelevinsky.
\newblock Cluster algebras. {II}. {F}inite type classification.
\newblock {\em Invent. Math.}, 154(1):63--121, 2003.

\bibitem{GSV1}
M.~Gekhtman, M.~Shapiro, and A.~Vainshtein.
\newblock Cluster algebras and {P}oisson geometry.
\newblock {\em Mosc. Math. J.}, 3(3):899--934, 1199, 2003.
\newblock \{Dedicated to Vladimir Igorevich Arnold on the occasion of his 65th
  birthday\}.

\bibitem{GSV}
M.~Gekhtman, M.~Shapiro, and A.~Vainshtein.
\newblock {\em Cluster algebras and {P}oisson geometry}, volume 167 of {\em
  Mathematical Surveys and Monographs}.
\newblock American Mathematical Society, Providence, RI, 2010.

\bibitem{gekhtman2012cluster}
M.~Gekhtman, M.~Shapiro, and A.~Vainshtein.
\newblock Cluster structures on simple complex {L}ie groups and
  {B}elavin-{D}rinfeld classification.
\newblock {\em Mosc. Math. J.}, 12(2):293--312, 460, 2012.

\bibitem{gekhtman2013arxiv}
M.~Gekhtman, M.~Shapiro, and A.~Vainshtein.
\newblock Exotic cluster structures on ${SL}_n$: the cremmer-gervais case.
\newblock {\em arXiv preprint arXiv:1307.1020}, 2013.

\bibitem{gekhtman2013exotic}
M.~Gekhtman, M.~Shapiro, and A.~Vainshtein.
\newblock Cremmer-{G}ervais cluster structure on {$SL_n$}.
\newblock {\em Proc. Natl. Acad. Sci. USA}, 111(27):9688--9695, 2014.

\bibitem{Scott}
J.~S. Scott.
\newblock Grassmannians and cluster algebras.
\newblock {\em Proc. London Math. Soc. (3)}, 92(2):345--380, 2006.

\end{thebibliography}

\end{document}